\documentclass[11pt,english]{article}
\usepackage{lmodern}
\usepackage[T1]{fontenc}
\usepackage[latin9]{inputenc}
\usepackage{geometry}
\usepackage{babel}
\usepackage{float}
\usepackage{amsmath}
\usepackage{amsthm}
\usepackage{amssymb}
\usepackage{graphicx}
\usepackage[unicode=true,pdfusetitle,
 bookmarks=true,bookmarksnumbered=false,bookmarksopen=false,
 breaklinks=false,pdfborder={0 0 0},pdfborderstyle={},backref=false,colorlinks=false]
 {hyperref}

\makeatletter

\providecommand{\tabularnewline}{\\}

\numberwithin{equation}{section}
\numberwithin{figure}{section}
\theoremstyle{definition}
\newtheorem{defn}{\protect\definitionname}[section]
\theoremstyle{remark}
\newtheorem{rem}{\protect\remarkname}[section]
\theoremstyle{plain}
\newtheorem{thm}{\protect\theoremname}[section]
\theoremstyle{plain}
\newtheorem{prop}{\protect\propositionname}[section]
\theoremstyle{plain}
\newtheorem{cor}{\protect\corollaryname}[section]
\theoremstyle{plain}
\newtheorem{lem}{\protect\lemmaname}[section]
\theoremstyle{definition}
\newtheorem{example}{\protect\examplename}[section]

\@ifundefined{date}{}{\date{}}
\makeatother

\providecommand{\corollaryname}{Corollary}
\providecommand{\definitionname}{Definition}
\providecommand{\examplename}{Example}
\providecommand{\lemmaname}{Lemma}
\providecommand{\propositionname}{Proposition}
\providecommand{\remarkname}{Remark}
\providecommand{\theoremname}{Theorem}

\begin{document}
\title{Asymptotic universal moment matching properties of normal distributions}
\author{Xuan Liu\thanks{Nomura Securities, Hong Kong SAR. Email: \protect\href{mailto:chamonixliu@163.com}{chamonixliu@163.com}}}
\maketitle
\begin{abstract}
Moment matching is an easy-to-implement and usually effective method
to reduce variance of Monte Carlo simulation estimates. On the other
hand, there is no guarantee that moment matching will always reduce
simulation variance for general integration problems at least asymptotically,
i.e. when the number of samples is large. We study the characterization
of conditions on a given underlying distribution $X$ under which
asymptotic variance reduction is guaranteed for a general integration
problem $\mathbb{E}[f(X)]$ when moment matching techniques are applied.
We show that a sufficient and necessary condition for such asymptotic
variance reduction property is $X$ being a normal distribution. Moreover,
when $X$ is a normal distribution, formulae for efficient estimation
of simulation variance for (first and second order) moment matching
Monte Carlo are obtained. These formulae allow estimations of simulation
variance as by-products of the simulation process, in a way similar
to variance estimations for plain Monte Carlo. Moreover, we propose
non-linear moment matching schemes for any given continuous distribution
such that asymptotic variance reduction is guaranteed.
\end{abstract}

\section{\label{sec:}Introduction}

Monte Carlo simulation is a general and versatile method for solving
numerical integration problems, particularly when these problems can
be formulated as expectations of random variables. It has been widely
applied across various fields, including operations research, statistical
physics, and engineering. In finance, Monte Carlo methods have become
a valuable tool for financial derivatives pricing and risk management,
especially in the pricing of path dependent derivatives, for which
closed-form solutions typically do not exist. Compared to grid based
methods such as finite difference methods, Monte Carlo simulation
is relatively less susceptible to the curse-of-dimension. However,
given a fixed computational budget, its convergence rate deteriorates
as the problem dimension increases. Variance reduction is therefore
of particular importance when applying Monte Carlo simulation to problems
like financial derivatives pricing due to their high-dimension nature
arising from both the number of underlying assets and time discretization
of driving stochastic differential equations. Several variance reduction
techniques have been proposed and extensively studied in the literature,
including antithetic variables, control variate, moment matching,
importance sampling, quasi Monte Carlo, and others. Each of these
methods has its own strengths and limitations. For instance, antithetic
variable and moment matching are easy to implement. However, they
do not guarantee lower simulation variance for general integration
problems; their effectiveness depends on the properties of the underlying
random variable and the integrand. The control variate method, in
contrast, always reduces variance when the optimal weight is used.
But the extent of variance reduction achieved depends on the choice
of a suitable control variate, which typically requires some knowledge
of the integrand. Additionally, extra computation effort is needed
to obtain a satisfactory estimation of the optimal weight. Most of
the aforementioned methods exhibit the standard Monte Carlo convergence
rate of $\text{O}(N^{-1/2})$. In contrast, quasi Monte Carlo methods
often achieve a faster convergence rate: $\text{O}(N^{-3/2+\epsilon})$
convergence rate was proved in \cite{Ow97} for any $\epsilon>0$
provided that the integrand function is sufficiently smooth. This
improvement comes at the cost of introducing dependence among samples,
which complicates the accurate quantification of simulation error.
It is worth noting that moment matching also introduces sample dependence,
making variance estimation in moment matching Monte Carlo less straightforward
than in plain Monte Carlo. See \cite{FH83,BBG97,RSS85} for more discussions
on the antithetic variable method, and \cite{DS98,DGS01} for the
moment matching method. For discussions from the theoretical aspect
of quasi Monte Carlo methods, see \cite{Ow92,Ow97,OR21}. A systematic
and comprehensive discussion on variance reduction techniques is given
in \cite{Gla04}.

For a general Monte Carlo simulation system---such as a derivatives
pricing library used in the finance industry---it is important to
determine in advance whether applying a particular variance reduction
technique will indeed lower the simulation variance. Accurate estimation
of simulation error is also critical in finance, not only for practical
reliability but because it is a regulatory requirement. In this paper,
we are concerned with the variance reduction property of moment matching
methods for general integrand functions, and efficient estimation
of simulation error. Our main results are Theorem \ref{thm:}, Theorem
\ref{thm:-1}, Proposition \ref{prop:-1}, and Proposition \ref{prop:-2}.
Theorem \ref{thm:} and Theorem \ref{thm:-1} state that moment matching
with respect to an underlying distribution always asymptotically reduces
variance (i.e. when samples size is large) if and only if this distribution
is a normal distribution. Proposition \ref{prop:-1} and Proposition
\ref{prop:-2} provide efficient estimations of the simulation variances
for moment matching Monte Carlo when the underlying distribution is
normal. A direct implication of Theorem \ref{thm:} and Theorem \ref{thm:-1}
is that, for any continuous underlying distribution, applying a non-linear
moment matching scheme guarantees asymptotic variance reduction for
general integrand functions.

In the rest of this section, we introduce notations and definitions
employed throughout the paper. Suppose that $X=(X_{1},\dots,X_{n})^{\text{T}}$
is an $n$-dimensional random vector with $\mathbb{E}(|X|^{2})<\infty$
and non-singular covariance matrix. We denote its expectation by
\[
\mathbb{E}(X_{i})=\mu_{i},\;i=1,\dots,n,
\]
and its covariance matrix by
\[
\Sigma=(\Sigma_{ij})_{ij},\;\Sigma_{ij}=\text{Cov}(X_{i},X_{j}),\;1\le i,j\le n.
\]
Consider the problem of estimating the integral $\mathbb{E}[f(X)]$
for a given integrand function $f(x)$. The plain Monte Carlo method
proceeds by simulating $N$ i.i.d. samples $\{X(k):1\le k\le N\}$
from the distribution of $X$, and computing the sample mean $I_{N}=N^{-1}\sum_{k=1}^{N}f(X(k))$.
The estimator $I_{N}$ satisfies $\lim_{N\to\infty}I_{N}=\mathbb{E}[f(X)]$
a.s., $\mathbb{E}(I_{N})=\mathbb{E}[f(X)]$, and $\text{Var}(I_{N})=N^{-1}\text{Var}[f(X)]$.
When certain moments of $X$ are known in advance, the moment matching
Monte Carlo leverages this information to adjust the sample values.
By enforcing consistency with the known moments, one can hopefully
reduce the multiplicative constant in $\text{Var}(I_{N})=\text{O}(N^{-1})$,
thereby improving the estimator\textquoteright s efficiency. 

The following definitions of moment matching were first proposed and
tested by \cite{Bar95}. A similar method, known as empirical martingale
simulation, was proposed in \cite{DS98}. The empirical martingale
simulation method can be viewed as moment matching in the $\log$-space.
It is worth pointing out that, due to the presence of a drift term
in the exponential martingale, the empirical martingale simulation
partially matches the first and the second moments of the driving
Brownian motion. 
\begin{defn}
\label{def:-1}The \emph{first order moment matching} estimator is
given by
\begin{equation}
\tilde{I}_{N}^{(1)}=\frac{1}{N}\sum_{k=1}^{N}f(\tilde{X}^{(1)}(k)),\label{eq:-1}
\end{equation}
where 
\begin{align}
\tilde{X}^{(1)}(k) & =X(k)-\bar{X}+\mu,\label{eq:-31}
\end{align}
and
\begin{equation}
\bar{X}=\frac{1}{N}\sum_{k=1}^{N}X(k).\label{eq:-2}
\end{equation}
\end{defn}
Clearly, the sample mean of the first order moment matched samples
$\tilde{X}^{(1)}(k)$ equals the known expectation $\mu$. It is known
that, under certain technical assumptions, $\lim_{N\to\infty}\tilde{I}_{N}^{(1)}=\lim_{N\to\infty}\mathbb{E}(\tilde{I}_{N}^{(1)})=\mathbb{E}[f(X)]$
a.s. 
\begin{defn}
The \emph{second order moment matching} estimator is given by
\begin{equation}
\tilde{I}_{N}^{(2)}=\frac{1}{N}\sum_{k=1}^{N}f(\tilde{X}^{(2)}(k)),\label{eq:-32}
\end{equation}
where
\begin{equation}
\tilde{X}^{(2)}(k)=\Sigma^{1/2}\bar{\Sigma}^{-1/2}[X(k)-\bar{X}]+\mu,\label{eq:-33}
\end{equation}
and 
\begin{equation}
\bar{\Sigma}=\frac{1}{N}\sum_{k=1}^{N}X(k)X(k)^{\mathrm{T}}-\bar{X}\bar{X}^{\text{T}}.\label{eq:-34}
\end{equation}
\end{defn}
Both the sample mean and the sample variance of the second order moment
matched samples $\tilde{X}^{(2)}(k)$ match the known expectation
$\mu$ and variance $\Sigma$. Note that we adopt the biased estimator
(\ref{eq:-34}) for the sample variance $\bar{\Sigma}$ just for convenience.
The conclusions in this paper remain the same if an unbiased estimator
is used for $\bar{\Sigma}$. 

We are mainly concerned with the following asymptotic universal moment
matching property. Throughout this paper, the support of a function
$f$ will be denoted by $\text{supp}(f)$, and the interior of a subset
$E\subseteq\mathbb{R}^{n}$ will be denoted by $E^{o}$ .
\begin{defn}
\label{def:}A distribution $X$ with density function $p(x)$ is
said to have the \emph{(first order) asymptotic universal moment matching
property}, if 
\begin{equation}
\mathrm{Var}\Big[\frac{1}{N}\sum_{k=1}^{N}f(\tilde{X}^{(1)}(k))\Big]\le A_{N}\mathrm{Var}\Big[\frac{1}{N}\sum_{k=1}^{N}f(X(k))\Big]=\frac{A_{N}}{N}\cdot\mathrm{Var}[f(X)],\label{eq:-3}
\end{equation}
for any smooth function $f$ with compact support $\text{supp}(f)\subseteq\text{supp}(p)^{o}$,
where the constant $A_{N}$ satisfies 
\begin{equation}
\lim_{N\to\infty}A_{N}=1,\label{eq:-4}
\end{equation}
and the strict inequality 
\begin{equation}
\lim_{N\to\infty}N\mathrm{Var}\Big[\frac{1}{N}\sum_{k=1}^{N}f(\tilde{X}^{(1)}(k))\Big]<\mathrm{Var}[f(X)],\label{eq:-61}
\end{equation}
holds for some $f$. Similarly, $X$ is said to have the \emph{(second
order) asymptotic universal moment matching property}, if
\begin{equation}
\mathrm{Var}\Big[\frac{1}{N}\sum_{k=1}^{N}f(\tilde{X}^{(2)}(k))\Big]\le A_{N}\mathrm{Var}\Big[\frac{1}{N}\sum_{k=1}^{N}f(X(k))\Big]=\frac{A_{N}}{N}\cdot\mathrm{Var}[f(X)],\label{eq:-35}
\end{equation}
for any smooth function $f$ with compact support $\text{supp}(f)\subseteq\text{supp}(p)^{o}$,
where the constant $A_{N}$ satisfies (\ref{eq:-4}).
\end{defn}
\begin{rem}
\label{rem:-3}The condition (\ref{eq:-61}) is to exclude the less
interesting cases where first order moment matching is asymptotically
equivalent to plain Monte Carlo; as we will see in Section \ref{sec:-1},
this can only happens for uniform distributions. On the other hand,
for second order moment matching, as we will see by Example \ref{exa:-1}
in Section \ref{sec:-3}, such less interesting cases are automatically
eliminated by the condition (\ref{eq:-35}).
\end{rem}
We make some final comments on the requirement of $\text{supp}(f)\subseteq\text{supp}(p)^{o}$
in Definition \ref{def:}. Consider a distribution $X$ with density
$p(x)$ supported in the interval $B=(-1,1)$. The definition of $f$
outside of $B$ is immaterial for plain Monte Carlo since all samples
of $X$ will be in $B$. However, for moment matching Monte Carlo,
it is possible to have $\tilde{X}(k)\not\in B$, which means that
the variance of moment matching Monte Carlo depends on specific extensions
of $f$ outside of $\text{supp}(p)$. Therefore, it is natural to
restrict the integrand functions $f$ to those supported in the interior
of $\text{supp}(p)$ when considering asymptotic universal moment
matching properties. On the other hand, we will see that such restriction
can be removed once we have proved that normal distributions are the
only distribution satisfying the asymptotic universal moment matching
properties.

\section{\label{sec:-6}Main results}

We summarize the main results in this section. Proofs of the main
results will be deferred to Section \ref{sec:-1} and Section \ref{sec:-3}.
\begin{thm}
\label{thm:}(i) Suppose that $X$ is a normal distribution. Then
$X$ has the first order asymptotic universal moment matching property.
More specifically, for any smooth $f$ with compact support,
\begin{equation}
\mathrm{Var}\Big[\frac{1}{N}\sum_{k=1}^{N}f(\tilde{X}^{(1)}(k))\Big]=\frac{A_{N}}{N}\big(\mathrm{Var}[f(X)]-\mathbb{E}[\partial f(X)]\Sigma\,\mathbb{E}[\partial f(X)]^{\mathrm{T}}\big),\label{eq:-18}
\end{equation}
with $A_{N}=1+\mathrm{O}(N^{-1/2})$. 

(ii) Conversely, suppose that $X$ is a continuous distribution with
differentiable density $p(x)$. If $X$ satisfies the first order
asymptotic universal moment matching property, then $X$ is a normal
distribution.
\end{thm}
\begin{rem}
\label{rem:}Given a normal distribution, (\ref{eq:-18}) tells about
when (first order) moment matching gives strict variance reduction.
Let $F(x)=\mathbb{E}[f(x+X)]$ for any $x\in\mathbb{R}^{n}$. Then
(\ref{eq:-18}) implies that the first order moment matching strictly
reduces variance when $\partial F(0)=\mathbb{E}[\partial f(X)]\not=0$,
i.e. when $x=0$ is a non-critical point of $F(x)$.
\end{rem}
The variance formula (\ref{eq:-18}) can be written in a different
form which generalizes to rough integrand functions $f$ satisfying
a mild integrability condition.
\begin{prop}
\label{prop:-1}Suppose that $X$ is a normal distribution with zero
mean and non-singular $\mathrm{Var}(X)=\Sigma$. Let $f$ be a Lebesgue
measurable function such that $\mathbb{E}[(1+|X|)^{4}f(X)^{2}]<\infty$.
Then
\begin{equation}
\mathrm{Var}\Big[\frac{1}{N}\sum_{k=1}^{N}f(\tilde{X}^{(1)}(k))\Big]=\frac{A_{N}}{N}\big(\mathrm{Var}[f(X)]-\mathbb{E}[Xf(X)]^{\mathrm{T}}\Sigma^{-1}\mathbb{E}[Xf(X)]\big),\label{eq:-27}
\end{equation}
with $A_{N}=1+\mathrm{O}(N^{-1/2})$. 
\end{prop}
It is easily seen that, when $f$ is a smooth function with compact
support, (\ref{eq:-27}) follows from (\ref{eq:-18}) and integration
by parts. For rough integrand $f$, it is tempting to approximate
a rough integrand $f$ by smooth functions $f_{\epsilon}$ and apply
a density argument. However, such argument does not work because,
as we will see in Section \ref{sec:-1}, the coefficient $A_{N}$
in (\ref{eq:-18}) involves second order derivatives of $f$, which
is not bounded even under the $L^{1}$ norm. In Section \ref{sec:-1},
we will prove Proposition \ref{prop:-1} by a duality argument.

For the second order moment matching, we have parallel results.
\begin{thm}
\label{thm:-1}(i) Suppose that $X$ is normal distribution. Then
$X$ has the second order asymptotic universal moment matching property.
Moreover, for any smooth $f$ with compact support,
\begin{equation}
\begin{aligned}\mathrm{Var}\Big[\frac{1}{N}\sum_{k=1}^{N}f(\tilde{X}^{(2)}(k))\Big] & =\frac{A_{N}}{N}\Big(\mathrm{Var}[f(X)]-\mathbb{E}[\partial f(X)]\Sigma\,\mathbb{E}[\partial f(X)]^{\mathrm{T}}-\frac{1}{2}\mathrm{tr}\big[(\Sigma\,\mathbb{E}[\partial^{2}f(X)])^{2}\big]\Big)\end{aligned}
,\label{eq:-50}
\end{equation}
with $A_{N}=1+\text{O}(N^{-1/2})$.

(ii) Conversely, suppose that $X$ is a continuous distribution with
differentiable density $p(x)$. If $X$ satisfies the second order
asymptotic universal moment matching property, then $X$ is a normal
distribution.
\end{thm}
We should point out that the last term $\mathrm{tr}\big[(\Sigma\,\mathbb{E}[\partial^{2}f(X)])^{2}\big]$
in (\ref{eq:-50}) is non-negative due to the identity $\mathrm{tr}\big[(\Sigma\,\mathbb{E}[\partial^{2}f(X)])^{2}\big]=\mathrm{tr}\big[(\Sigma^{1/2}\mathbb{E}[\partial^{2}f(X)]\Sigma^{1/2})^{2}\big]$.
\begin{rem}
\label{rem:-2}Let t $F(x)=\mathbb{E}[f(x+X)]$, $x\in\mathbb{R}^{n}$.
Then (\ref{eq:-50}) implies that the second order moment matching
strictly reduces variance when $\partial F(0)\not=0$ or $\partial^{2}F(0)\not=0$.
\end{rem}
\begin{prop}
\label{prop:-2}Suppose that $X$ is a normal distribution with zero
mean and non-singular $\mathrm{Var}(X)=\Sigma$. Let $f$ be a Lebesgue
measurable function such that $\mathbb{E}[(1+|X|)^{8}f(X)^{2}]<\infty$.
Then 
\begin{equation}
\begin{aligned}\mathrm{Var}\Big[\frac{1}{N}\sum_{k=1}^{N}f(\tilde{X}^{(2)}(k))\Big] & =\frac{A_{N}}{N}\Big(\mathrm{Var}[f(X)]-\mathbb{E}[Xf(X)]^{\mathrm{T}}\Sigma^{-1}\mathbb{E}[Xf(X)]\\
 & \quad-\frac{1}{2}\mathrm{tr}\big[\mathbb{E}\big((\Sigma^{-1}XX^{\mathrm{T}}-I)f(X)\big)^{2}\big]\Big),
\end{aligned}
\label{eq:-65}
\end{equation}
with $A_{N}=1+\mathrm{O}(N^{-1/2})$. 
\end{prop}
Similar to (\ref{eq:-50}), The term $\mathrm{tr}\big[\mathbb{E}\big((\Sigma^{-1}XX^{\mathrm{T}}-I)f(X)\big)^{2}\big]$
in (\ref{eq:-65}) is non-negative. As a corollary of (\ref{eq:-18})
and (\ref{eq:-50}), it is seen that second order moment matching
provides more variance reduction than first order moment matching
does. Another consequence of Proposition \ref{prop:-1} and Proposition
\ref{prop:-2} is that they allow efficient computation of simulation
variance.
\begin{cor}
\label{cor:}Suppose that $X$ is a normal distribution with mean
$\mu=0$ and variance $\Sigma=I$. Then the variance of the first
order moment matching Monte Carlo estimator can be estimated by
\begin{equation}
\begin{aligned}\mathrm{Var}\Big[\frac{1}{N}\sum_{k=1}^{N}f(\tilde{X}^{(1)}(k))\Big] & \approx\frac{1}{N}\Big[\frac{1}{N}\sum_{k=1}^{N}f(\tilde{X}^{(1)}(k))^{2}-\Big(\frac{1}{N}\sum_{k=1}^{N}f(\tilde{X}^{(1)}(k))\Big)^{2}\\
 & \quad-\Big(\frac{1}{N}\sum_{k=1}^{N}\tilde{X}^{(1)}(k)f(\tilde{X}^{(1)}(k))\Big)^{\text{T}}\Sigma^{-1}\Big(\frac{1}{N}\sum_{k=1}^{N}\tilde{X}^{(1)}(k)f(\tilde{X}^{(1)}(k))\Big)\Big].
\end{aligned}
\label{eq:-29}
\end{equation}
Moreover, the variance of the second order moment matching Monte Carlo
estimator can be estimated by
\begin{equation}
\begin{aligned}\mathrm{Var}\Big[\frac{1}{N}\sum_{k=1}^{N}f(\tilde{X}^{(2)}(k))\Big] & \approx\frac{1}{N}\Big\{\frac{1}{N}\sum_{k=1}^{N}f(\tilde{X}^{(2)}(k))^{2}-\Big(\frac{1}{N}\sum_{k=1}^{N}f(\tilde{X}^{(2)}(k))\Big)^{2}\\
 & \quad-\Big(\frac{1}{N}\sum_{k=1}^{N}\tilde{X}^{(2)}(k)f(\tilde{X}^{(2)}(k))\Big)^{\text{T}}\Sigma^{-1}\Big(\frac{1}{N}\sum_{k=1}^{N}\tilde{X}^{(2)}(k)f(\tilde{X}^{(2)}(k))\Big)\\
 & \quad-\frac{1}{2}\mathrm{tr}\Big[\Big(\frac{1}{N}\sum_{k=1}^{N}[\Sigma^{-1}\tilde{X}^{(2)}(k)\tilde{X}^{(2)}(k)^{\mathrm{T}}-I]f(\tilde{X}^{(2)}(k))\Big)^{2}\Big]\Big\}.
\end{aligned}
\label{eq:-92}
\end{equation}
\end{cor}
Note that in the process of the first order (or second order) moment
matching Monte Carlo, values of $\tilde{X}^{(1)}(k)$ and $f(\tilde{X}^{(1)}(k))$
(respectively $\tilde{X}^{(2)}(k)$ and $f(\tilde{X}^{(2)}(k))$)
are readily available. Therefore, Corollary \ref{cor:} yields an
estimation of the simulation error as a by-product of estimation of
$\mathbb{E}[f(X)]$ using moment matching. In other words, (\ref{eq:-29})
and (\ref{eq:-92}) requires only $\text{O}(1)$ Monte Carlo cycles
to compute. Here a Monte Carlo cycle means, for a fixed number of
samples $N$, the procedure of simulation of $N$ samples, followed
by $N$ evaluations of the integrand function. It is worth mentioning
that the matrix operations in (\ref{eq:-29}) and (\ref{eq:-92})
have only marginal computation cost: they are of roughly the same
cost as applying the correlation matrix to get correlated samples.
In contrast, formula (\ref{eq:-18}) in Theorem \ref{thm:} requires
$\text{O}(n)$ Monte Carlo cycles, and formula (\ref{eq:-50}) in
Theorem \ref{thm:-1} requires $\text{O}(N^{-2})$ cycles. This is
because the estimation of each expectation $\mathbb{E}[\partial_{i}f(X)]$
and $\mathbb{E}[\partial_{ij}^{2}f(X)]$, $1\le i,j\le n$ requires
$\text{O}(1)$ Monte Carlo cycles. Besides higher computation costs,
the effectiveness of formulae (\ref{eq:-18}) and (\ref{eq:-50})
is sensitive to the smoothness of the integrand. More specifically,
consider the example of $f$ being the Heaviside function $\chi_{[0,\infty)}$.
The expectation $\mathbb{E}[\partial f(X)]$ is an integral of the
Dirac delta function, for which an integrand smoothing technique is
usually required before applying Monte Carlo methods, at the cost
of introducing bias.

Another implication of Theorem \ref{thm:} is that for any continuous
random variable $Y$, there is a non-linear moment matching scheme
which guarantees asymptotic variance reduction. 
\begin{cor}
\label{cor:-1}Let $Y$ be a continuous one-dimensional random variable
with cumulative probability function $F_{Y}(y)=\mathbb{P}(Y\le y)$.
Let 
\[
\mathcal{N}(x)=\frac{1}{\sqrt{2\pi}}\int_{-\infty}^{x}e^{-z^{2}/2}dz,
\]
 be the cumulative probability function of the standard normal distribution,
and $X=\mathcal{N}^{-1}(F_{Y}(Y))$. Then $X$ is a standard normal
distribution. Moreover, for any bounded smooth function $f$, let
\begin{align}
\tilde{Y}^{(1)}(k) & =(F_{Y}^{-1}\mathcal{N})(X(k)-\bar{X}),\label{eq:-30}\\
\tilde{Y}^{(2)}(k) & =(F_{Y}^{-1}\mathcal{N})[\bar{\sigma}^{-1}(X(k)-\bar{X})]\label{eq:-93}
\end{align}
where $X(k)=\mathcal{N}^{-1}(F_{Y}(Y(k)))$ and 
\[
\bar{X}=\frac{1}{N}\sum_{k=1}^{N}X(k),\quad\bar{\sigma}^{2}=\frac{1}{N}\sum_{k=1}^{N}X(k)^{2}-\bar{X}^{2}.
\]
Then
\begin{align}
\tilde{I}_{N}^{(1)} & =\frac{1}{N}\sum_{k=1}^{N}f(\tilde{Y}^{(1)}(k)),\label{eq:-94}\\
\tilde{I}_{N}^{(2)} & =\frac{1}{N}\sum_{k=1}^{N}f(\tilde{Y}^{(2)}(k)),\label{eq:-95}
\end{align}
are estimators of $\mathbb{E}[f(Y)]$ such that $\lim_{N\to\infty}\tilde{I}_{N}^{(1)}=\lim_{N\to\infty}\mathbb{E}(\tilde{I}_{N}^{(1)})=\mathbb{E}[f(Y)]$
a.s., and
\[
\lim_{N\to\infty}N\mathrm{Var}(\tilde{I}_{N}^{(2)})\le\lim_{N\to\infty}N\mathrm{Var}(\tilde{I}_{N}^{(1)})\le\mathrm{Var}[f(Y)].
\]
\end{cor}
From the implementation perspective, simulation of random samples
of $Y$ proceeds by sampling random uniform distributions and then
converting these samples into samples of $Y$. Corollary \ref{cor:-1}
suggests an intermediate step involving a transformation to normal
distributions, which is a common step in stochastic processes modeling.
The moment matching technique is then applied to samples from this
intermediate normal distribution, rather than directly to samples
of $Y$. The non-linear moment matching schemes in Corollary \ref{cor:-1}
can be applied to continuous multi-dimensional distributions by applying
the one-dimensional scheme to the conditional distributions one by
one.

\section{\label{sec:-1}Proof of main results: first order moment matching}

For any sequence $\{\xi_{N}\}_{N}$ of random variables, we shall
adopt the notation $\xi_{N}=\text{O}_{r}(N^{-m})$ for $m>0$, if
\begin{equation}
\mathbb{E}(|\xi_{N}|^{p})^{1/p}\le c_{p}N^{-m},\label{eq:-5}
\end{equation}
for any $1\le p\le r$, where $c_{p}>0$ is a constant depending on
$p$ and the sequence $\{\xi_{N}\}_{N}$. Similarly, we denote $\xi_{N}=\text{o}_{r}(N^{-m})$
if
\begin{equation}
\lim_{N\to\infty}N^{m}\mathbb{E}(|\xi_{N}|^{p})^{1/p}=0,\label{eq:}
\end{equation}
for any $1\le p\le r$. 
\begin{lem}
\label{lem:}Suppose that $X$ is a continuous random vector with
$\mathbb{E}(|X|^{2})<\infty$. Let $f$ be a smooth function with
compact support. Then 
\begin{equation}
\begin{aligned}\mathrm{Var}\Big[\frac{1}{N}\sum_{k=1}^{N}f(\tilde{X}^{(1)}(k))\Big] & =\frac{1}{N}\mathrm{Var}[f(X)]-\frac{2}{N}\mathbb{E}[\partial f(X)]\mathbb{E}[(X-\mu)f(X)]\\
 & \quad+\frac{1}{N}\mathbb{E}[\partial f(X)]\Sigma\,\mathbb{E}[\partial f(X)]^{\text{T}}+\mathrm{o}(N^{-1}).
\end{aligned}
\label{eq:-20}
\end{equation}
If, in addition $\mathbb{E}(|X|^{4})<\infty$, then the remainder
in (\ref{eq:-20}) is of order $\text{O}(N^{-3/2})$.
\end{lem}
\begin{proof}
By replacing $f$ with $f-c$, we may assume $\mathbb{E}[f(X)]=0$.
Without loss of generality, we assume $\mathbb{E}(X)=0$. Clearly
$\bar{X}=\text{O}_{2}(N^{-1/2})$. Therefore, by Taylor's formula,
\[
f(\tilde{X}^{(1)}(k))=f(X(k))-\sum_{i}\partial_{i}f(\xi(k))\bar{X}_{i},
\]
where $|\xi(k)-X(k)|\le|\bar{X}|$. Note that, by the dominated convergence
theorem, $[\partial_{i}f(\xi(k))-\partial_{i}(X(k))]\bar{X}_{i}=\text{o}_{2}(N^{-1/2})$.
Therefore,
\begin{equation}
f(\tilde{X}^{(1)}(k))=f(X(k))-\sum_{i}\partial_{i}f(X(k))\bar{X}_{i}+\text{o}_{2}(N^{-1/2}).\label{eq:-7}
\end{equation}
and
\begin{equation}
\frac{1}{N}\sum_{k=1}^{N}f(\tilde{X}^{(1)}(k))=\frac{1}{N}\sum_{k=1}^{N}f(X(k))-\frac{1}{N}\sum_{k=1}^{N}\sum_{i}\partial_{i}f(X(k))\bar{X}_{i}+\text{o}_{2}(N^{-1/2}).\label{eq:-8}
\end{equation}
Note that, by $\mathbb{E}[f(X)]=0$, we have $\frac{1}{N}\sum_{k=1}^{N}f(X(k))=\text{O}_{2}(N^{-1/2})$,
which implies
\begin{equation}
\begin{aligned}\Big(\frac{1}{N}\sum_{k=1}^{N}f(\tilde{X}^{(1)}(k))\Big)^{2} & =\Big(\frac{1}{N}\sum_{k=1}^{N}f(X(k))\Big)^{2}+\Big(\frac{1}{N}\sum_{k=1}^{N}\sum_{i}\partial_{i}f(X(k))\bar{X}_{i}\Big)^{2}\\
 & \quad-2\sum_{i}\Big(\frac{1}{N}\sum_{k=1}^{N}f(X(k))\Big)\Big(\frac{1}{N}\sum_{k=1}^{N}\partial_{i}f(X(k))\bar{X}_{i}\Big)+\text{o}_{1}(N^{-1}).
\end{aligned}
\label{eq:-9}
\end{equation}
We now compute the expectation of individual terms on the right side
of (\ref{eq:-9}). For the first term,
\begin{equation}
\mathbb{E}\Big[\Big(\frac{1}{N}\sum_{k=1}^{N}f(X(k))\Big)^{2}\Big]=\frac{1}{N}\text{Var}[f(X)].\label{eq:-10}
\end{equation}
For the second term, we have 
\begin{equation}
\frac{1}{N}\sum_{k=1}^{N}\partial_{i}f(X(k))=\mathbb{E}[\partial_{i}f(X)]+\text{O}_{2}(N^{-1/2}),\label{eq:-14}
\end{equation}
by the central limit theorem. This gives
\begin{equation}
\begin{aligned}\mathbb{E}\Big[\Big(\frac{1}{N}\sum_{k=1}^{N}\sum_{i}\partial_{i}f(X(k))\bar{X}_{i}\Big)^{2}\Big] & =\mathbb{E}\Big[\Big(\sum_{i}\mathbb{E}[\partial_{i}f(X)]\bar{X}_{i}\Big)^{2}\Big]+\text{O}(N^{-3/2})\\
 & =\frac{1}{N}\sum_{i,j}\Sigma_{ij}\mathbb{E}[\partial_{i}f(X)]\mathbb{E}[\partial_{j}f(X)]+\text{O}(N^{-3/2}).
\end{aligned}
\label{eq:-11}
\end{equation}
Similarly,
\begin{equation}
\begin{aligned}\mathbb{E}\Big[\Big(\frac{1}{N}\sum_{k=1}^{N}f(X(k))\Big)\Big(\frac{1}{N}\sum_{k=1}^{N}\partial_{i}f(X(k))\bar{X}_{i}\Big)\Big] & =\mathbb{E}\Big[\Big(\frac{1}{N}\sum_{k=1}^{N}f(X(k))\Big)\big(\mathbb{E}[\partial_{i}f(X)]\bar{X}_{i}\big)\Big]+\text{O}(N^{-3/2})\\
 & =\mathbb{E}[\partial_{i}f(X)]\mathbb{E}[f(X(1))\bar{X}_{i}]+\text{O}(N^{-3/2})\\
 & =\frac{1}{N}\mathbb{E}[\partial_{i}f(X)]\mathbb{E}[f(X(1))X_{i}(1)]+\text{O}(N^{-3/2})\\
 & =\frac{1}{N}\mathbb{E}[\partial_{i}f(X)]\mathbb{E}[X_{i}f(X)]+\text{O}(N^{-3/2}).
\end{aligned}
\label{eq:-12}
\end{equation}
Combining (\ref{eq:-9}), \pageref{eq:-10}, (\ref{eq:-11}), (\ref{eq:-12})
gives
\begin{equation}
\begin{aligned}\mathbb{E}\Big[\Big(\frac{1}{N}\sum_{k=1}^{N}f(\tilde{X}^{(1)}(k))\Big)^{2}\Big] & =\frac{1}{N}\mathrm{Var}[f(X)]-\frac{2}{N}\sum_{i}\mathbb{E}[\partial_{i}f(X)]\mathbb{E}[X_{i}f(X)]\\
 & \quad+\frac{1}{N}\sum_{i,j}\Sigma_{ij}\mathbb{E}[\partial_{i}f(X)]\mathbb{E}[\partial_{j}f(X)]+\text{o}(N^{-1}).
\end{aligned}
\label{eq:-15}
\end{equation}
By (\ref{eq:-7}) and (\ref{eq:-14}),
\[
\frac{1}{N}\sum_{k=1}^{N}f(\tilde{X}^{(1)}(k))=\frac{1}{N}\sum_{k=1}^{N}f(X(k))+\sum_{i}\mathbb{E}[\partial_{i}f(X)]\bar{X}_{i}+\text{o}_{2}(N^{-1/2}),
\]
which, together with $\mathbb{E}[f(X)]=0$ and $\mathbb{E}(\bar{X})=0$,
implies that
\begin{equation}
\mathbb{E}\Big(\frac{1}{N}\sum_{k=1}^{N}f(\tilde{X}^{(1)}(k))\Big)=\text{o}(N^{-1/2}).\label{eq:-16}
\end{equation}
Now (\ref{eq:-20}) follows readily from (\ref{eq:-15}) and (\ref{eq:-16}).

Suppose in addition that $\mathbb{E}(|X|^{4})<\infty$. Then $\bar{X}=\text{O}_{4}(N^{-1/2})$.
To see (\ref{eq:-5}) for $p>2$, by the Burkholder--David--Gundy
inequality and Minkowski's inequality, 
\[
\begin{aligned}N^{p}\mathbb{E}(|\bar{X}|^{p}) & =\mathbb{E}\Big(\Big|\sum_{k=1}^{N}X(k)\Big|^{p}\Big)\\
 & \le c_{p}\mathbb{E}\Big(\Big|\sum_{k=1}^{N}X(k)^{2}\Big|^{p/2}\Big)\le c_{p}\Big(\sum_{k=1}^{N}\mathbb{E}[|X(k)|^{p}]^{2/p}\Big)^{p/2}
\end{aligned}
\]
which implies (\ref{eq:-5}). Therefore, by Taylor's formula,
\[
\begin{aligned}f(\tilde{X}^{(1)}(k)) & =f(X(k))-\sum_{i}\partial_{i}f(X(k))\bar{X}_{i}+\sum_{i,j}\partial_{ij}^{2}f(\xi(k))\bar{X}_{i}\bar{X}_{j}\\
 & =f(X(k))-\sum_{i}\partial_{i}f(X(k))\bar{X}_{i}+\text{O}_{2}(N^{-1}).
\end{aligned}
\]
The rest of the proof is similar to that of the above.
\end{proof}
\begin{example}
\label{exa:}As an application of Lemma \ref{lem:}, we show that
uniform distributions do not satisfy the first order asymptotic moment
matching property. Suppose for simplicity that $n=1$, and $p(x)=\frac{1}{2\sqrt{3}}\chi_{B}(x)$,
where $B=(-\sqrt{3},\sqrt{3})$. Then $\mathbb{E}(X)=0$ and $\mathbb{E}(X^{2})=1$.
Note that $\mathbb{E}[f^{\prime}(X)]=0$ for any smooth function $f$
with $\text{supp}(f)\subseteq B$. It follows from Lemma \ref{lem:}
that
\[
\begin{aligned}\lim_{N\to\infty}N\mathrm{Var}\Big[\frac{1}{N}\sum_{k=1}^{N}f(\tilde{X}^{(1)}(k))\Big] & =\mathrm{Var}[f(X)].\end{aligned}
\]
Therefore, we see that uniform distributions do not satisfy the condition
(\ref{eq:-61}). 
\end{example}
\begin{example}
\label{exa:-2}It is easy to find an example for which first order
moment matching increases variance instead. Let $X$ be the exponential
distribution $p(x)=e^{-x}\chi_{(0,\infty)}(x)$. Then $\mathbb{E}(X)=\text{Var}(X)=1$.
Let $f$ be a non-negative smooth function supported in $(0,1)$.
Integrating by parts shows $\mathbb{E}[f^{\prime}(X)]=\mathbb{E}[f(X)]$.
By Lemma \ref{lem:},
\[
\begin{aligned}\lim_{N\to\infty}N\mathrm{Var}\Big[\frac{1}{N}\sum_{k=1}^{N}f(\tilde{X}^{(1)}(k))\Big] & =\mathrm{Var}[f(X)]+3\mathbb{E}[f(X)]^{2}-2\mathbb{E}[f(X)]\mathbb{E}[Xf(X)].\end{aligned}
\]
Since $f\ge0$ and $\text{supp}(f)\subseteq(0,1)$, we have $\mathbb{E}[Xf(X)]\le\mathbb{E}[f(X)]$.
Therefore, the right side of the above is at least $\mathrm{Var}[f(X)]+\mathbb{E}[f(X)]^{2}>\text{Var}[f(X)]$
when $f\not=0$.
\end{example}
We are now in a position to prove the asymptotic universal moment
matching property of normal distributions. 
\begin{proof}[Proof of Theorem \ref{thm:}]
(i) Suppose that $X$ is a normal distribution with density 
\[
p(x)=[2\pi\det(\Sigma)]^{-n/2}e^{-(x-\mu)^{\text{T}}\Sigma^{-1}(x-\mu)/2}
\]
By $\partial p(x)=(\mu-x)^{\text{T}}\Sigma^{-1}p(x)$ and integration
by part,
\begin{equation}
\begin{aligned}\mathbb{E}[\partial f(X)] & =\int_{\mathbb{R}^{n}}\partial f(x)p(x)dx=-\int_{\mathbb{R}^{n}}f(x)\partial p(x)dx\\
 & =\int_{-\infty}^{\infty}(x-\mu)^{\text{T}}\Sigma^{-1}f(x)p(x)dx=\mathbb{E}[(X-\mu)f(X)]^{\text{T}}\Sigma^{-1}.
\end{aligned}
\label{eq:-19}
\end{equation}
By Lemma \ref{lem:},
\begin{equation}
\begin{aligned}\mathrm{Var}\Big[\frac{1}{N}\sum_{k=1}^{N}f(\tilde{X}^{(1)}(k))\Big] & =\frac{1}{N}\mathrm{Var}[f(X)]-\frac{1}{N}\mathbb{E}[\partial f(X)]^{\text{T}}\Sigma\,\mathbb{E}[\partial f(X)]+\text{O}(N^{-3/2}).\end{aligned}
\label{eq:-28}
\end{equation}
for any smooth function $f$ with compact support. This implies (\ref{eq:-18})
with $A_{N}=1+\text{O}(N^{-1/2})$. Therefore, $X$ satisfies the
first order asymptotic universal moment matching property. 

(ii) For the converse, suppose that $X$ satisfies the first order
asymptotic universal moment matching property. We need to show that
$X$ is a normal distribution. Suppose first that $n=1$, that is,
$X$ is a one dimensional random variable. By Lemma \ref{lem:},
\begin{equation}
\begin{aligned}\mathrm{Var}\Big[\frac{1}{N}\sum_{k=1}^{N}f(\tilde{X}^{(1)}(k))\Big] & =\frac{A_{N}}{N}\big(\mathrm{Var}[f(X)]-2\mathbb{E}[f^{\prime}(X)]\mathbb{E}[(X-\mu)f(X)]+\sigma^{2}\mathbb{E}[f^{\prime}(X)]^{2}\big),\end{aligned}
\label{eq:-21}
\end{equation}
for any smooth function $f$ with compact support, where $\sigma^{2}=\text{Var}(X)$
and $\lim_{N\to\infty}A_{N}=1$. Since $X$ satisfies the asymptotic
universal moment matching property, it follows from (\ref{eq:-21})
that
\begin{equation}
\mathbb{E}[f^{\prime}(X)]\mathbb{E}[(X-\mu)f(X)]\ge0,\label{eq:-22}
\end{equation}
for any smooth function $f$ with compact support and and $\text{supp}(f)\subseteq\text{supp}(p)^{o}$.
Integrating by parts gives
\[
\mathbb{E}[f^{\prime}(X)]=-\int_{-\infty}^{\infty}f(x)p^{\prime}(x)dx.
\]
The inequality (\ref{eq:-22}) gives
\begin{equation}
\int_{-\infty}^{\infty}f(x)p^{\prime}(x)dx\cdot\int_{-\infty}^{\infty}(\mu-x)f(x)p(x)dx\ge0.\label{eq:-23}
\end{equation}
Let $\phi(x)$ be a smooth function supported in $(-1,1)$ and $\int\phi(x)dx=1$.
Let $\phi_{\epsilon}(x)=\epsilon^{-1}\phi(x/\epsilon)$. For any $x_{1},x_{2}\in\text{supp}(p)$,
setting $f(x)=a_{1}\phi_{\epsilon}(x_{1}-x)+a_{2}\phi_{\epsilon}(x_{2}-x)$
in (\ref{eq:-23}) and letting $\epsilon\to0$ gives
\begin{equation}
[a_{1}p^{\prime}(x_{1})+a_{2}p^{\prime}(x_{2})]\cdot[a_{1}(\mu-x_{1})p(x_{1})+a_{2}(\mu-x_{2})p(x_{2})]\le0.\label{eq:-25}
\end{equation}
By Lemma \ref{lem:-2} below,
\[
\frac{p^{\prime}(x_{1})}{(\mu-x_{1})p(x_{1})}=\frac{p^{\prime}(x_{2})}{(\mu-x_{2})p(x_{2})},
\]
which implies
\begin{equation}
\frac{p^{\prime}(x)}{(\mu-x)p(x)}=c_{1},\label{eq:-24}
\end{equation}
for some constant $c_{1}$. Solving the differential equation (\ref{eq:-24})
gives $p(x)=c_{2}e^{-c_{1}(x-\mu)^{2}/2}$. Note that Example \ref{exa:}
implies that $p^{\prime}\not=0$ and consequently $c_{1}\not=0$.
It then follows from $\int p(x)dx=1$ and $\text{Var}(X)=\sigma^{2}$
that $c_{1}=\sigma^{-2}$, $c_{2}=(2\pi\sigma^{2})^{-1/2}$. Therefore,
$X$ is a normal distribution with mean $\mu$ and variance $\sigma^{2}$.

For multi-dimensional case, it is temping to reduce to the one dimensional
case. This can be done if we assume a stronger condition that (\ref{eq:-3})
is valid for any bounded smooth function $f$. To see this, let $\lambda=(\lambda_{1},\dots,\lambda_{n})$
and $Y=\sum_{i}\lambda_{i}X_{i}$. Clearly, (\ref{eq:-3}) for $X$
(for bounded smooth functions) implies (\ref{eq:-3}) for $Y$ (for
smooth functions with compact support).\footnote{Note that $f(\lambda_{1}x_{1}+\cdots+\lambda_{n}x_{n})$ does not
have compact support in $\mathbb{R}^{n}$ even if $f$ has compact
support in $\mathbb{R}$.} By the one dimensional case, $Y$ is a normal distribution. Since
$\mathbb{E}(Y)=\lambda^{\text{T}}\mu$ and $\text{Var}(Y)=\lambda^{\text{T}}\Sigma\,\lambda$,
we obtain that
\begin{equation}
\mathbb{E}[e^{\text{i}(\lambda_{1}X_{1}+\cdots+\lambda_{n}X_{n})}]=\mathbb{E}[e^{\text{i}Y}]=e^{\text{i}\lambda^{\text{T}}\mu-\frac{1}{2}\lambda^{\text{T}}\Sigma\,\lambda}.\label{eq:-13}
\end{equation}
It follows from (\ref{eq:-13}) that $X$ is a normal distribution
with mean $\mu$ and covariance matrix $\Sigma$. A proof without
assuming (\ref{eq:-3}) for all bounded smooth functions is given
in \nameref{sec:-4}.
\end{proof}
We now turn to the proof of Proposition \ref{prop:-1}.
\begin{proof}[Proof of Proposition \ref{prop:-1}]
We may assume $\mathbb{E}[f(X)]=0$. Suppose for convenience of notation
that $n=1$, $\mathbb{E}(X)=0$ and $\mathbb{E}(X^{2})=1$. Clearly,
\begin{equation}
\begin{aligned}\mathbb{E}\Big[\Big(\frac{1}{N}\sum_{k=1}^{N}f(\tilde{X}^{(1)}(k))\Big)^{2}\Big] & =\frac{1}{N}\mathbb{E}[f(\tilde{X}^{(1)}(1))^{2}]+\Big(1-\frac{1}{N}\Big)\mathbb{E}[f(\tilde{X}^{(1)}(1))f(\tilde{X}^{(1)}(2))]\\
 & =\frac{1}{N}E+\Big(1-\frac{1}{N}\Big)E^{\prime}.
\end{aligned}
\label{eq:-91}
\end{equation}
We compute the expectations $E$ and $E^{\prime}$. For any $x=(x_{1},\dots,x_{N})\in\mathbb{R}^{N}$,
denote 
\[
\bar{x}=N^{-1}\sum_{k=1}^{N}x_{k}.
\]
Define the linear transformation $T:\mathbb{R}^{N}\to\mathbb{R}^{N}$
by 
\[
Tx=(x_{1}-\bar{x},x_{2}-\bar{x},x_{3},\dots,x_{N}\big).
\]
Then, $\det[\partial T(x)]=1-2N^{-1}$, and the inverse of $T$ is
given by 
\begin{equation}
T^{-1}y=(y_{1}+\bar{y},y_{2}+\bar{y},y_{3},\dots,y_{N}),\quad y\in\mathbb{R}^{N}.\label{eq:-85}
\end{equation}
By change of variable,
\[
\begin{aligned}E & =\int_{\mathbb{R}^{N}}f((Tx)_{1})^{2}p(x)dx\\
 & =\frac{1}{1-2N^{-1}}\int_{\mathbb{R}^{N}}f(y_{1})^{2}p(T^{-1}y)dy\\
 & =\frac{1}{1-2N^{-1}}\int_{\mathbb{R}^{N}}f(x_{1})^{2}\frac{p(T^{-1}x)}{p(x)}p(x)dx\\
 & =\frac{1}{1-2N^{-1}}\mathbb{E}\Big[f(X(1))^{2}\frac{p(T^{-1}X)}{p(X)}\Big],
\end{aligned}
\]
where $p(x)=(2\pi)^{-N/2}e^{-|x|^{2}/2}$ is the density function
of the $N$-dimensional standard normal distribution. Similarly,
\[
\begin{aligned}E^{\prime} & =\int_{\mathbb{R}^{N}}f((Tx)_{1})f((Tx)_{2})p(x)dx\\
 & =\frac{1}{1-2N^{-1}}\int_{\mathbb{R}^{N}}f(y_{1})f(y_{2})p(T^{-1}y)dy\\
 & =\frac{1}{1-2N^{-1}}\int_{\mathbb{R}^{N}}f(x_{1})f(x_{2})\frac{p(T^{-1}x)}{p(x)}p(x)dx\\
 & =\frac{1}{1-2N^{-1}}\mathbb{E}\Big[f(X(1))f(X(2))\frac{p(T^{-1}X)}{p(X)}\Big].
\end{aligned}
\]
By (\ref{eq:-85}), $\partial p(x)=-xp(x)$, $\partial^{2}p(x)=(xx^{\text{T}}-I)p(x)$,
and Taylor's formula,
\[
\begin{aligned}\frac{p(T^{-1}X)}{p(X)} & =1-\sum_{k=1}^{2}X(k)\bar{X}+\frac{1}{2}\sum_{k=1}^{2}[X(k)^{2}-1]\bar{X}^{2}\\
 & \quad+X(1)X(2)\bar{X}^{2}+\text{O}_{2}(N^{-3/2}).
\end{aligned}
\]
Therefore, by $\mathbb{E}[f(X)]=0$ and symmetry,
\[
\begin{aligned}E^{\prime} & =-\frac{2}{1-2N^{-1}}\mathbb{E}[X(1)f(X(1))f(X(2))\bar{X}]\\
 & \quad+\frac{1}{1-2N^{-1}}\mathbb{E}[(X(1)^{2}-1)f(X(1))f(X(2))\bar{X}^{2}]\\
 & \quad+\frac{1}{1-2N^{-1}}\mathbb{E}[X(1)f(X(1))X(2)f(X(2))\bar{X}^{2}]+\text{O}_{2}(N^{-3/2}).
\end{aligned}
\]
By $\mathbb{E}[f(X)]=0$ again, it is easily seen that 
\[
\begin{aligned} & \mathbb{E}[X(1)f(X(1))f(X(2))\bar{X}]=\frac{1}{N}\mathbb{E}[Xf(X)]^{2},\\
 & \mathbb{E}[(X(1)^{2}-1)f(X(1))f(X(2))\bar{X}^{2}]=\text{O}(N^{-2}),\\
 & \mathbb{E}[X(1)f(X(1))X(2)f(X(2))\bar{X}^{2}]=\frac{1}{N}\mathbb{E}[Xf(X)]^{2}+\text{O}(N^{-2}).
\end{aligned}
\]
Therefore,
\begin{equation}
E^{\prime}=-\frac{1}{N}\mathbb{E}[Xf(X)]^{2}+\text{O}(N^{-3/2}).\label{eq:-86}
\end{equation}
By similar argument, it can be shown that
\begin{equation}
\begin{aligned}E & =\mathbb{E}[f(X)^{2}]+\text{O}_{2}(N^{-1/2})\end{aligned}
,\label{eq:-89}
\end{equation}
and
\begin{equation}
\begin{aligned}\mathbb{E}\Big(\frac{1}{N}\sum_{k=1}^{N}f(\tilde{X}^{(1)}(k))\Big) & =\frac{1}{1-2N^{-1}}\mathbb{E}\Big[f(X(1))\frac{p(T^{-1}X)}{p(X)}\Big]+\text{O}(N^{-3/2})\\
 & =\frac{1}{1-2N^{-1}}\mathbb{E}[f(X(1))(X(1)+X(2))\bar{X}]+\text{O}(N^{-1})\\
 & =\text{O}(N^{-1}).
\end{aligned}
\label{eq:-90}
\end{equation}
Combining (\ref{eq:-91}), (\ref{eq:-86}), (\ref{eq:-89}), (\ref{eq:-90})
completes the proof for $n=1$ and $\text{Var}(X)=1$.

The proof of $n>1$ and $\text{Var}(X)=I$ is similar to the above.
For general correlated normal distribution $X$, (\ref{eq:-27}) follows
from change of variable $Y=\Sigma^{-1/2}X$.
\end{proof}
We finish this section by the following simple lemma used in the proof
of Theorem \ref{thm:}.
\begin{lem}
\label{lem:-2}Let $H$ be a Hilbert space and $\beta_{1},\beta_{2}\in H$.
If the quadratic form $Q(\alpha)=\langle\alpha,\beta_{1}\rangle\langle\alpha,\beta_{2}\rangle$
is positive semi-definite, then $\beta_{1},\beta_{2}$ are linearly
dependent. 
\end{lem}
\begin{proof}
We may assume $\dim(H)>1$ and $\beta_{i}\not=0$, $i=1,2$. Suppose
for contradiction that $\beta_{1},\beta_{2}$ are linearly independent.
Then there exist $\alpha_{i}\in H$ such that $\langle\alpha_{1},\beta_{1}\rangle=\langle\alpha_{2},\beta_{2}\rangle=1$
and $\langle\alpha_{1},\beta_{2}\rangle=\langle\alpha_{2},\beta_{1}\rangle=0$.
Let $\alpha=\alpha_{1}-\alpha_{2}$. Now $Q(\alpha)=-\langle\alpha_{1},\beta_{1}\rangle\langle\alpha_{2},\beta_{2}\rangle=-1$
contradicts with the positive semi-definiteness of $Q$. This proves
the lemma.
\end{proof}

\section{\label{sec:-3}Proof of main results: second order moment matching}

The following lemma will be used several times in this section. Its
proof is by direct computation and application of the central limit
theorem. For completeness, the proof of Lemma \ref{lem:-4} will be
given in \nameref{sec:-5}. 
\begin{lem}
\label{lem:-4}Suppose that $X$ is a continuous random vector with
$\mathbb{E}(|X|^{4})<\infty$, $\mathbb{E}(X)=0$, and $\mathrm{Var}(X)=I$.
Let $g(x,y)$ be a Borel measurable function on $\mathbb{R}^{n}\times\mathbb{R}^{n}$
such that 
\[
\mathbb{E}[|g(X(1),X(2))|^{2}(1+|X(1)|+|X(2)|)^{8}]<\infty.
\]
Then
\begin{align}
 & \mathbb{E}[g(X(1),X(2))\bar{X}_{i}\bar{X}_{j}]=\frac{1}{N}\mathbb{E}[g(X(1),X(2))]\delta_{ij}+\mathrm{O}(N^{-2}),\label{eq:-73}\\
 & \begin{aligned}\mathbb{E}[g(X(1),X(2))(\bar{\Sigma}-I)_{ij}] & =\frac{1}{N}\mathbb{E}[g(X(1),X(2))(X_{i}(1)X_{j}(1)+X_{i}(2)X_{j}(2))]\\
 & \quad-\frac{3}{N}\mathbb{E}[g(X(1),X(2))]\delta_{ij}+\mathrm{O}(N^{-2}),
\end{aligned}
\label{eq:-74}\\
 & \mathbb{E}[g(X(1),X(2))(\bar{\Sigma}-I)_{ij}\bar{X}_{p}]=\frac{1}{N}\mathbb{E}[g(X(1),X(2))]\mathbb{E}(X_{i}X_{j}X_{p})+\mathrm{O}(N^{-3/2}),\label{eq:-75}\\
 & \begin{aligned}\mathbb{E}[g(X(1),X(2))(\bar{\Sigma}-I)_{ij}(\bar{\Sigma}-I)_{pq}] & =\frac{1}{N}\mathbb{E}[g(X(1),X(2))][\mathbb{E}(X_{i}X_{j}X_{p}X_{q})-\delta_{ij}\delta_{pq}]\\
 & \quad+\mathrm{O}(N^{-3/2}).
\end{aligned}
\label{eq:-76}
\end{align}
If, in addition, $\mathbb{E}[g(X,y)]=0$ for all $y\in\mathbb{R}^{n}$,
then 
\begin{equation}
\begin{aligned}\mathbb{E}[g(X(1),X(2))(\bar{\Sigma}^{1/2}-I)_{ij}] & =\frac{1}{2N}\mathbb{E}[g(X(1),X(2))X_{i}(1)X_{j}(1)]+\mathrm{O}(N^{-3/2}).\end{aligned}
\label{eq:-77}
\end{equation}
\end{lem}
Similar to the procedure in Section \ref{sec:-1}, we start with the
following variance expansion for second order moment matching.
\begin{lem}
\label{lem:-1}Suppose that $X$ is a continuous random vector with
$\mathbb{E}(|X|^{4})<\infty$, $\mathbb{E}(X)=0$, and $\mathrm{Var}(X)=I$.
Let $f$ be a smooth function with compact support. Then 
\begin{equation}
\begin{aligned} & \mathrm{Var}\Big[\frac{1}{N}\sum_{k=1}^{N}f(\tilde{X}^{(2)}(k))\Big]\\
 & =\frac{1}{N}\mathrm{Var}[f(X)]+\frac{1}{4N}\sum_{i,j,p,q}\mathbb{E}[\partial_{i}f(X)X_{j}]\mathbb{E}[\partial_{p}f(X)X_{q}][\mathbb{E}(X_{i}X_{j}X_{p}X_{q})-\delta_{ij}\delta_{pq}]\\
 & \quad+\frac{1}{N}\sum_{i}\mathbb{E}[\partial_{i}f(X)]^{2}+\frac{1}{N}\sum_{i,j}\mathbb{E}[\partial_{i}f(X)X_{j}]\mathbb{E}[f(X)(\delta_{ij}-X_{i}X_{j})]\\
 & \quad-\frac{2}{N}\sum_{i}\mathbb{E}[\partial_{i}f(X)]\mathbb{E}[f(X)X_{i}]+\frac{1}{N}\sum_{i,j,p}\mathbb{E}[\partial_{i}f(X)X_{j}]\mathbb{E}[\partial_{p}f(X)]\mathbb{E}(X_{i}X_{j}X_{p})\\
 & \quad+\mathrm{O}(N^{-3/2}),
\end{aligned}
\label{eq:-42}
\end{equation}
where $\delta_{ij}=1$ if $i=j$, and $\delta_{ij}=0$ otherwise.
\end{lem}
\begin{proof}
We may assume $\mathbb{E}[f(X)]=0$. Similar to that of Lemma \ref{lem:},
by Taylor's formula and the central limit theorem,
\[
\begin{aligned}\frac{1}{N}\sum_{k=1}^{N}f(\tilde{X}^{(2)}(k)) & =\frac{1}{N}\sum_{k=1}^{N}f(X(k))+\frac{1}{N}\sum_{k=1}^{N}\sum_{i,j}\partial_{i}f(X(k))(\bar{\Sigma}^{-1/2}-I)_{ij}X_{j}(k)\\
 & \quad-\frac{1}{N}\sum_{k=1}^{N}\sum_{i,j}\partial_{i}f(X(k))(\bar{\Sigma}^{-1/2})_{ij}\bar{X}_{j}+\text{O}_{2}(N^{-1}),\\
 & =\frac{1}{N}\sum_{k=1}^{N}f(X(k))+\sum_{i,j}\mathbb{E}[\partial_{i}f(X)X_{j}](\bar{\Sigma}^{-1/2}-I)_{ij}\\
 & \quad-\sum_{i,j}\mathbb{E}[\partial_{i}f(X)](\bar{\Sigma}^{-1/2})_{ij}\bar{X}_{j}+\text{O}_{2}(N^{-1}).
\end{aligned}
\]
Since $\mathbb{E}(X)=0$ and $\text{Var}(X)=I$, we have
\begin{equation}
\bar{\Sigma}_{ij}=\frac{1}{N}\sum_{k=1}^{N}X_{i}(k)X_{j}(k)+\text{O}(N^{-1}),\quad1\le i,j\le n,\label{eq:-47}
\end{equation}
which implies $\bar{\Sigma}=I+\text{O}(N^{-1/2})$. Moreover, $\bar{\Sigma}^{-1/2}-I=(\bar{\Sigma}+\bar{\Sigma}^{1/2})^{-1}(I-\bar{\Sigma})=\frac{1}{2}(I-\bar{\Sigma})+\text{O}(N^{-1})$.
Therefore,
\begin{equation}
\begin{aligned}\frac{1}{N}\sum_{k=1}^{N}f(\tilde{X}^{(2)}(k)) & =\frac{1}{N}\sum_{k=1}^{N}f(X(k))+\frac{1}{2}\sum_{i,j}\mathbb{E}[\partial_{i}f(X)X_{j}](I-\bar{\Sigma})_{ij}\\
 & -\sum_{i}\mathbb{E}[\partial_{i}f(X)]\bar{X}_{i}+\text{O}_{2}(N^{-1}).
\end{aligned}
\label{eq:-49}
\end{equation}
By (\ref{eq:-47}), (\ref{eq:-49}) and $\mathbb{E}[f(X)]=0$,
\[
\begin{aligned}\mathbb{E}\Big(\frac{1}{N}\sum_{k=1}^{N}f(\tilde{X}^{(2)}(k))\Big) & =\text{O}_{2}(N^{-1}),\end{aligned}
\]
and consequently,
\begin{equation}
\begin{aligned} & \mathrm{Var}\Big[\frac{1}{N}\sum_{k=1}^{N}f(\tilde{X}^{(2)}(k))\Big]\\
 & =\mathbb{E}\Big[\Big(\frac{1}{N}\sum_{k=1}^{N}f(\tilde{X}^{(2)}(k))\Big)^{2}\Big]+\text{O}_{2}(N^{-2})\\
 & =\frac{1}{N}\text{Var}[f(X)]+\frac{1}{4}\mathbb{E}\Big[\Big(\sum_{i,j}\mathbb{E}[\partial_{i}f(X)X_{j}](I-\bar{\Sigma})_{ij}\Big)^{2}\Big]+\mathbb{E}\Big[\Big(\sum_{i}\mathbb{E}[\partial_{i}f(X)]\bar{X}_{i}\Big)^{2}\Big]\\
 & \quad+\sum_{i,j}\mathbb{E}[\partial_{i}f(X)X_{j}]\mathbb{E}[f(X(1))(I-\bar{\Sigma})_{ij}]-2\sum_{i}\mathbb{E}[\partial_{i}f(X)]\mathbb{E}[f(X(1))\bar{X}_{i}]\\
 & \quad-\mathbb{E}\Big[\Big(\sum_{i,j}\mathbb{E}[\partial_{i}f(X)X_{j}](I-\bar{\Sigma})_{ij}\Big)\Big(\sum_{i}\mathbb{E}[\partial_{i}f(X)]\bar{X}_{i}\Big)\Big]+\text{O}(N^{-3/2})\\
 & =\frac{1}{N}\text{Var}[f(X)]+\frac{1}{4}E_{1}+E_{2}+E_{3}-2E_{4}-E_{5}+\text{O}(N^{-3/2}).
\end{aligned}
\label{eq:-26}
\end{equation}

We compute the expectations $E_{k},1\le k\le5$ one by one. For $E_{1}$,
by (\ref{eq:-76}),
\[
\begin{aligned}\mathbb{E}[(I-\bar{\Sigma})_{ij}(I-\bar{\Sigma})_{pq}] & =\frac{1}{N}[\mathbb{E}(X_{i}X_{j}X_{p}X_{q})-\delta_{ij}\delta_{pq}]+\text{O}(N^{-3/2}).\end{aligned}
\]
Therefore,
\begin{equation}
\begin{aligned}E_{1} & =\frac{1}{N}\sum_{i,j,p,q}\mathbb{E}[\partial_{i}f(X)X_{j}]\mathbb{E}[\partial_{p}f(X)X_{q}][\mathbb{E}(X_{i}X_{j}X_{p}X_{q})-\delta_{ij}\delta_{pq}]+\text{O}(N^{-3/2}).\end{aligned}
\label{eq:-43}
\end{equation}
For $E_{2}$ and $E_{4}$, we have
\begin{equation}
\begin{aligned}E_{2} & =\sum_{i,j}\mathbb{E}[\partial_{i}f(X)]\mathbb{E}[\partial_{j}f(X)]\mathbb{E}(\bar{X}_{i}\bar{X}_{j})=\frac{1}{N}\sum_{i}\mathbb{E}[\partial_{i}f(X)]^{2},\end{aligned}
\label{eq:-44}
\end{equation}
and
\begin{equation}
\begin{aligned}E_{4} & =\frac{1}{N}\sum_{i}\mathbb{E}[\partial_{i}f(X)]\mathbb{E}[f(X)X_{i}]\end{aligned}
.\label{eq:-46}
\end{equation}
For $E_{3}$, by (\ref{eq:-74}) and $\mathbb{E}[f(X)]=0$, 
\begin{equation}
\begin{aligned}E_{3} & =\frac{1}{N}\sum_{i,j}\mathbb{E}[\partial_{i}f(X)X_{j}]\mathbb{E}[f(X)(\delta_{ij}-X_{i}X_{j})]+\text{O}(N^{-2}),\end{aligned}
\label{eq:-45}
\end{equation}
For $E_{5}$, by (\ref{eq:-75}),
\begin{equation}
\begin{aligned}E_{5} & =\sum_{i,j,p}\mathbb{E}[\partial_{i}f(X)X_{j}]\mathbb{E}[\partial_{p}f(X)]\mathbb{E}[(I-\bar{\Sigma})_{ij}\bar{X}_{p}]\\
 & =-\frac{1}{N}\sum_{i,j,p}\mathbb{E}[\partial_{i}f(X)X_{j}]\mathbb{E}[\partial_{p}f(X)]\mathbb{E}(X_{i}X_{j}X_{p})+\text{O}(N^{-3/2}).
\end{aligned}
\label{eq:-48}
\end{equation}
Combining (\ref{eq:-26}), (\ref{eq:-43}), (\ref{eq:-44}), (\ref{eq:-46}),
(\ref{eq:-45}), (\ref{eq:-48}) completes the proof.
\end{proof}
\begin{example}
\label{exa:-1}As an application of Lemma \ref{lem:-1}, we show that
uniform distributions do not satisfy the second order asymptotic universal
moment matching property. Suppose for simplicity that $n=1$ and $p(x)=\frac{1}{2\sqrt{3}}\chi_{B}(x)$.
Then $\mathbb{E}(X)=0$, $\mathbb{E}(X^{2})=1$. For any smooth function
$f$ with $\text{supp}(f)\subseteq B$, we have $\mathbb{E}[f^{\prime}(X)]=0$
and $\mathbb{E}[f^{\prime}(X)X]=-\mathbb{E}[f(X)]$ by integrating
by parts. By Lemma \ref{lem:-1} and some simple calculation, 
\begin{equation}
\begin{aligned}\lim_{N\to\infty}N\mathrm{Var}\Big[\frac{1}{N}\sum_{k=1}^{N}f(\tilde{X}^{(2)}(k))\Big] & =\mathrm{Var}[f(X)]+\mathbb{E}[f(X)X^{2}]-\frac{4}{5}\mathbb{E}[f(X)]^{2}.\end{aligned}
\label{eq:-62}
\end{equation}
Let $\{f_{j}\}_{j}$ be a sequence of smooth functions such that $|f_{j}|\le1$
on $B$, $f_{j}(x)=1$ on $B_{j}$ and $\text{supp}(f_{j})\subseteq B_{j+1}$,
where $B_{j}=(1-2^{-j})B=[-\sqrt{3}(1-2^{-j}),\sqrt{3}(1-2^{-j})]$.
By (\ref{eq:-62}) and the dominated convergence theorem,
\[
\lim_{j\to\infty}\lim_{N\to\infty}\Big[N\mathrm{Var}\Big[\frac{1}{N}\sum_{k=1}^{N}f_{j}(\tilde{X}^{(2)}(k))\Big]-\mathrm{Var}[f_{j}(X)]\Big]=\mathbb{E}(X^{2})-\frac{4}{5}=\frac{1}{5}>0.
\]
This implies that uniform distributions do not satisfy the second
order asymptotic universal moment matching property.
\end{example}
We now proceed to the proof of Theorem \ref{thm:-1}.
\begin{proof}[Proof of Theorem \ref{thm:-1}]
(i) Suppose first that $\mathbb{E}(X)=0$ and $\text{Var}(X)=I$.
Clearly, $\mathbb{E}(X_{i}X_{j}X_{p}X_{q})=0$ if $i<j,p<q$ and $(i,j)\not=(p,q)$.
Therefore,
\[
\begin{aligned} & \sum_{i,j,p,q}\mathbb{E}[\partial_{i}f(X)X_{j}]\mathbb{E}[\partial_{p}f(X)X_{q}][\mathbb{E}(X_{i}X_{j}X_{p}X_{q})-\delta_{ij}\delta_{pq}]\\
 & =4\sum_{i<j}\sum_{p<q}\mathbb{E}[\partial_{i}f(X)X_{j}]\mathbb{E}[\partial_{p}f(X)X_{q}]\mathbb{E}(X_{i}X_{j}X_{p}X_{q})\\
 & \quad+4\sum_{i<j}\sum_{p}\mathbb{E}[\partial_{i}f(X)X_{j}]\mathbb{E}[\partial_{p}f(X)X_{p}]\mathbb{E}(X_{i}X_{j}X_{p}^{2})\\
 & \quad+\sum_{i,p}\mathbb{E}[\partial_{i}f(X)X_{i}]\mathbb{E}[\partial_{p}f(X)X_{p}][\mathbb{E}(X_{i}^{2}X_{p}^{2})-1]\\
 & =4\sum_{i<j}\mathbb{E}[\partial_{i}f(X)X_{j}]^{2}\mathbb{E}(X_{i}^{2}X_{j}^{2})+\sum_{i}\mathbb{E}[\partial_{i}f(X)X_{i}]^{2}[\mathbb{E}(X_{i}^{4})-1]\\
 & =2\sum_{i,j}\mathbb{E}[\partial_{i}f(X)X_{j}]^{2},
\end{aligned}
\]
where for the second equality, we used 
\[
\mathbb{E}(X_{i}X_{j}X_{p}X_{q})=0,\quad i<j,\;p<q,\;(i,j)\not=(p,q),
\]
and $\mathbb{E}(X_{i}X_{j}X_{p}^{2})=0$ if $i\not=j$. Hence, by
Lemma \ref{lem:-1} and $\mathbb{E}(X_{i}X_{j}X_{p})=0$,
\begin{equation}
\begin{aligned} & \mathrm{Var}\Big[\frac{1}{N}\sum_{k=1}^{N}f(\tilde{X}^{(2)}(k))\Big]\\
 & =\frac{1}{N}\mathrm{Var}[f(X)]+\frac{1}{2N}\sum_{i,j}\mathbb{E}[\partial_{i}f(X)X_{j}]^{2}+\frac{1}{N}\sum_{i}\mathbb{E}[\partial_{i}f(X)]^{2}\\
 & \quad+\frac{1}{N}\sum_{i,j}\mathbb{E}[\partial_{i}f(X)X_{j}]\mathbb{E}[f(X)(\delta_{ij}-X_{i}X_{j})]\\
 & \quad-\frac{2}{N}\sum_{i}\mathbb{E}[\partial_{i}f(X)]\mathbb{E}[f(X)X_{i}]+\mathrm{O}(N^{-3/2}).
\end{aligned}
\label{eq:-53}
\end{equation}
Let $p(x)=(2\pi)^{-n/2}e^{-|x|^{2}/2}$ by the density function of
$X$. Integration by parts gives 
\begin{equation}
\mathbb{E}[\partial_{ij}^{2}f(X)]=\int\partial_{ij}^{2}f(x)p(x)dx=\int\partial_{i}f(x)x_{j}p(x)dx=\mathbb{E}[\partial_{i}f(X)X_{j}].\label{eq:-51}
\end{equation}
Integrating by parts again,
\begin{equation}
\begin{aligned}\mathbb{E}[\partial_{i}f(X)X_{j}] & =\int\partial_{i}f(x)x_{j}p(x)dx\\
 & =\int f(x)(x_{i}x_{j}-\delta_{ij})p(x)dx=\mathbb{E}[f(X)(X_{i}X_{j}-\delta_{ij})].
\end{aligned}
\label{eq:-52}
\end{equation}
It follows from (\ref{eq:-53}), (\ref{eq:-51}), (\ref{eq:-52})
that
\[
\begin{aligned}\mathrm{Var}\Big[\frac{1}{N}\sum_{k=1}^{N}f(\tilde{X}^{(2)}(k))\Big]=\frac{A_{N}}{N}\Big(\mathrm{Var}[f(X)]-\sum_{i}\mathbb{E}[\partial_{i}f(X)]^{2}-\frac{1}{2}\sum_{i,j}\mathbb{E}[\partial_{ij}^{2}f(X)]^{2}\Big),\end{aligned}
\]
with $A_{N}=1+\text{O}(N^{-1/2})$, or in matrix notation,
\begin{equation}
\mathrm{Var}\Big[\frac{1}{N}\sum_{k=1}^{N}f(\tilde{X}^{(2)}(k))\Big]=\frac{A_{N}}{N}\Big(\mathrm{Var}[f(X)]-\mathbb{E}[\partial f(X)]\mathbb{E}[\partial f(X)]^{\text{T}}-\frac{1}{2}\text{tr}\big(\mathbb{E}[\partial^{2}f(X)]^{2}\big)\Big).\label{eq:-54}
\end{equation}

For the general case of correlated normal distributions, (\ref{eq:-50})
follows from applying (\ref{eq:-54}) to $Y=\Sigma^{-1/2}\,(X-\mu)$
and $g(Y)=f(\Sigma^{1/2}Y+\mu)$ and the identity
\[
\text{tr}\big[\big(\Sigma^{1/2}\mathbb{E}[\partial^{2}f(X)]\Sigma^{1/2}\big)^{2}\big]=\text{tr}\big[\big(\Sigma\,\mathbb{E}[\partial^{2}f(X)]\big)^{2}\big].
\]

(ii) Suppose that $X$ satisfies the second order asymptotic universal
moment matching property, we show that $X$ is a normal distribution.
We give the proof of this conclusion for $n=1$. The passage from
$n=1$ to $n>1$ is similar to the proof of Theorem \ref{thm:}, (ii)
with the stronger assumption that (\ref{eq:-35}) holds for all bounded
smooth functions. A proof for $n>1$ without this stronger assumption
is similar to the proof of Proposition \ref{prop:}. 

Suppose now $n=1$, and without loss of generality, $\mathbb{E}(X)=0$
and $\mathbb{E}(X^{2})=1$. By Lemma \ref{lem:-1},
\begin{equation}
\begin{aligned} & \mathrm{Var}\Big[\frac{1}{N}\sum_{k=1}^{N}f(\tilde{X}^{(2)}(k))\Big]\\
 & =\frac{1}{N}\mathrm{Var}[f(X)]+\frac{1}{4N}\mathbb{E}[f^{\prime}(X)X]^{2}[\mathbb{E}(X^{4})-1]+\frac{1}{N}\mathbb{E}[f^{\prime}(X)]^{2}\\
 & \quad+\frac{1}{N}\mathbb{E}[f^{\prime}(X)X]\mathbb{E}[f(X)(1-X^{2})]-\frac{2}{N}\mathbb{E}[f^{\prime}(X)]\mathbb{E}[f(X)X]\\
 & \quad+\frac{1}{N}\mathbb{E}[f^{\prime}(X)X]\mathbb{E}[f^{\prime}(X)]\mathbb{E}(X^{3})+\text{O}(N^{-3/2}).
\end{aligned}
\label{eq:-57}
\end{equation}
The key observation is that
\begin{equation}
\frac{1}{4}\mathbb{E}[f^{\prime}(X)X]^{2}[\mathbb{E}(X^{4})-1]+\mathbb{E}[f^{\prime}(X)X]\mathbb{E}[f^{\prime}(X)]\mathbb{E}(X^{3})+\mathbb{E}[f^{\prime}(X)]^{2}\ge0.\label{eq:-55}
\end{equation}
To see this, note that, by $\mathbb{E}(X)=0$, $\mathbb{E}(X^{2})=1$
and H\"{o}lder's inequality,
\begin{equation}
\mathbb{E}(X^{3})^{2}=\mathbb{E}[X(X^{2}-1)]^{2}\le\mathbb{E}(X^{2})\mathbb{E}[(X^{2}-1)^{2}]=\mathbb{E}(X^{4})-1.\label{eq:-56}
\end{equation}
Therefore, (\ref{eq:-55}) follows from (\ref{eq:-56}) and 
\[
\begin{aligned} & \frac{1}{4}\mathbb{E}[f^{\prime}(X)X]^{2}[\mathbb{E}(X^{4})-1]+\mathbb{E}[f^{\prime}(X)X]\mathbb{E}[f^{\prime}(X)]\mathbb{E}(X^{3})+\mathbb{E}[f^{\prime}(X)]^{2}\\
 & =\frac{1}{4}\mathbb{E}[f^{\prime}(X)X]^{2}[\mathbb{E}(X^{4})-1-\mathbb{E}(X^{3})^{2}]+\Big(\frac{1}{2}\mathbb{E}[f^{\prime}(X)X]\mathbb{E}(X^{3})+\mathbb{E}[f^{\prime}(X)]\Big)^{2}.
\end{aligned}
\]
By (\ref{eq:-57}) and (\ref{eq:-55}), we see that the second order
asymptotic universal moment matching property of $X$ implies
\begin{equation}
\mathbb{E}[f^{\prime}(X)X]\mathbb{E}[f(X)(1-X^{2})]-2\mathbb{E}[f^{\prime}(X)]\mathbb{E}[f(X)X]\le0,\label{eq:-58}
\end{equation}
for any smooth function $f$ with compact support $\text{supp}(f)\subseteq\text{supp}(p)^{o}$.
Integrating by parts gives
\begin{equation}
\mathbb{E}[f^{\prime}(X)]=-\int_{-\infty}^{\infty}f(x)p^{\prime}(x)dx,\label{eq:-63}
\end{equation}
and
\begin{equation}
\mathbb{E}[f^{\prime}(X)X]=-\int_{-\infty}^{\infty}f(x)[p(x)+xp^{\prime}(x)]dx.\label{eq:-64}
\end{equation}
Therefore, (\ref{eq:-58}) can be written as
\[
\int_{-\infty}^{\infty}f(x)[p(x)+xp^{\prime}(x)]dx\cdot\int_{-\infty}^{\infty}f(x)(1-x^{2})p(x)dx-2\int_{-\infty}^{\infty}f(x)p^{\prime}(x)dx\cdot\int_{-\infty}^{\infty}f(x)xp(x)dx\ge0,
\]
Let $\phi(x)$ be a smooth function supported in $(-1,1)$ and $\int\phi(x)dx=1$.
Denote $\phi_{\epsilon}(x)=\epsilon^{-1}\phi(x/\epsilon)$ for $\epsilon>0$.
Since Let $u\not=0$ be a point in $\text{supp}(p)$ such that $|u|\not=1$.
For any $x_{1},x_{2}\in\text{supp}(p)$ and any $a_{1},a_{2},b\in\mathbb{R}$,
setting $f(x)=\sum_{i=1}^{2}a_{i}\phi_{\epsilon}(x_{i}-x)+b\phi_{\epsilon}(u-x)$
in (\ref{eq:-58}) and letting $\epsilon\to0$ gives
\begin{equation}
\begin{aligned} & 2\Big(\sum_{i=1}^{2}a_{i}p^{\prime}(x_{i})+bp^{\prime}(u)\Big)\Big(\sum_{i=1}^{2}a_{i}x_{i}p(x_{i})+bup(u)\Big)\\
 & \le\Big(\sum_{i=1}^{2}a_{i}[p(x_{i})+x_{i}p^{\prime}(x_{i})]+b[p(u)+up^{\prime}(u)]\Big)\Big(\sum_{i=1}^{2}a_{i}(1-x_{i}^{2})p(x_{i})+b(1-u^{2})p(u)\Big).
\end{aligned}
\label{eq:-59}
\end{equation}
Setting 
\[
b=-\sum_{i=1}^{2}a_{i}\dfrac{(1-x_{i}^{2})p(x_{i})}{(1-u^{2})p(u)},
\]
in (\ref{eq:-59}) gives that
\[
\Big(\sum_{i=1}^{2}a_{i}\Big(p^{\prime}(x_{i})-\dfrac{p^{\prime}(u)}{(1-u^{2})p(u)}(1-x_{i}^{2})p(x_{i})\Big)\Big)\Big(\sum_{i=1}^{2}a_{i}\Big(x_{i}-\dfrac{u}{(1-u^{2})}(1-x_{i}^{2})\Big)p(x_{i})\Big)\le0
\]
\begin{equation}
\Big(\sum_{i=1}^{2}a_{i}[p^{\prime}(x_{i})-\xi(1-x_{i}^{2})p(x_{i})]\Big)\Big(\sum_{i=1}^{2}a_{i}[x_{i}-\eta(1-x_{i}^{2})]p(x_{i})\Big)\le0,\label{eq:-60}
\end{equation}
for any $a_{1},a_{2}\in\mathbb{R}$, where $\xi=p^{\prime}(u)/[(1-u^{2})p(u)]$
and $\eta=u/(1-u^{2})$ are fixed constants. It follows from (\ref{eq:-60})
and Lemma \ref{lem:-2} that
\[
\dfrac{p^{\prime}(x_{1})-\xi(1-x_{1}^{2})p(x_{1})}{[x_{1}-\eta(1-x_{1}^{2})]p(x_{1})}=\dfrac{p^{\prime}(x_{2})-\xi(1-x_{2}^{2})p(x_{2})}{[x_{2}-\eta(1-x_{2}^{2})]p(x_{2})},
\]
which implies that
\[
\dfrac{p^{\prime}(x)-\xi(1-x^{2})p(x)}{[x-\eta(1-x^{2})]p(x)}=-c_{1},\quad x\in\text{supp}(p),
\]
for some constant $c_{1}$. Solving this differential equation gives
$p(x)=c_{2}e^{-c_{1}x^{2}/2-c_{1}(\xi-\eta)(x-x^{3}/3)}$ for some
$c_{2}$. Note that $\partial p\not=0$ by revoking the conclusion
of Example \ref{exa:-1}. Therefore, $c_{1}\not=0$. Moreover, it
follows from $\mathbb{E}(X^{2})<\infty$ that $\xi-\eta=0$. This
implies that $X$ is a normal distribution.
\end{proof}
Let us turn to the proof of Proposition \ref{prop:-2}. The proof
is similar to that of Proposition \ref{prop:-1}, with some modification
to account for the non-linearity of the second order moment matching
transformation. More specifically, the inverse of $T^{-1}$ in the
proof of Proposition \ref{prop:-1} will be replaced by an approximate
inverse operator $S$.
\begin{proof}[Proof of Proposition \ref{prop:-2}]
We may assume that $\mathbb{E}[f(X)]=0$. Suppose for convenience
of notation that $n=1$, $\mathbb{E}(X)=0$ and $\mathbb{E}(X^{2})=1$.
We compute each term of the following expansion
\begin{equation}
\begin{aligned}\mathbb{E}\Big[\Big(\frac{1}{N}\sum_{k=1}^{N}f(\tilde{X}^{(2)}(k))\Big)^{2}\Big] & =\frac{1}{N}\mathbb{E}[f(\tilde{X}^{(2)}(1))^{2}]+\Big(1-\frac{1}{N}\Big)\mathbb{E}[f(\tilde{X}^{(2)}(1))f(\tilde{X}^{(2)}(2))]\\
 & =\frac{1}{N}E+\Big(1-\frac{1}{N}\Big)E^{\prime}.
\end{aligned}
\label{eq:-84}
\end{equation}
For any $x=(x_{1},\dots,x_{N})\in\mathbb{R}^{N}$, denote 
\[
\bar{x}=N^{-1}\sum_{k=1}^{N}x_{k},\quad\bar{\sigma}(x)^{2}=\frac{1}{N}\sum_{k=1}^{N}x_{k}^{2}-\bar{x}^{2}.
\]
Define $T:\mathbb{R}^{N}\to\mathbb{R}^{N}$ by 
\begin{equation}
Tx=\big(\bar{\sigma}(x)^{-1}(x_{1}-\bar{x}),\bar{\sigma}(x)^{-1}(x_{2}-\bar{x}),x_{3},\dots,x_{N}\big).\label{eq:-17}
\end{equation}
By direction computation,
\begin{equation}
\det[\partial T(x)]=\bar{\sigma}(x)^{-2}\big(1-N^{-1}(2+(Tx)_{1}^{2}+(Tx)_{2}^{2})+N^{-2}[(Tx)_{1}-(Tx)_{2}]^{2}\big).\label{eq:-67}
\end{equation}
For sufficiently large $N$, we may assume $\det[\partial T(x)]>0$
therefore $T^{-1}$ is well defined.\footnote{This can be made rigorous by truncating the integrals to the region
$\{x:|(Tx)_{1}|<L,|(Tx)_{2}|<L\}$ for sufficiently large $L=L(N)$,
and by the rapidly decreasing property of the density function $p(x)$.} Let $Y$ be the random vector defined by $Y=TX$. Then $\tilde{X}(k)=Y(k)$
for $k=1,2$. By (\ref{eq:-17}) and simple algebra,
\[
\bar{X}=\bar{Y}+N^{-1}[\bar{\sigma}(X)-1][Y(1)+Y(2)]+2N^{-1}\bar{X}=\bar{Y}+\text{O}(N^{-3/2}),
\]
and
\begin{equation}
\begin{aligned}\bar{\sigma}(X)^{2} & =\frac{1}{N}\sum_{k=1}^{N}Y_{k}^{2}-\bar{X}^{2}+N^{-1}(\bar{\sigma}(X)^{2}-1)[Y(1)^{2}+Y(2)^{2}]+2N^{-1}\bar{X}[X(1)+X(2)-\bar{X}]\\
 & =\frac{1}{N}\sum_{k=1}^{N}Y_{k}^{2}-\bar{X}^{2}+\text{O}(N^{-3/2})=\frac{1}{N}\sum_{k=1}^{N}Y_{k}^{2}-\bar{Y}^{2}+\text{O}(N^{-3/2})=\bar{\sigma}(Y)^{2}+\text{O}(N^{-3/2}).
\end{aligned}
\label{eq:-66}
\end{equation}
By (\ref{eq:-17}) again, for $k=1,2$,
\[
\begin{aligned}X(k) & =\bar{\sigma}(X)Y(k)+\frac{N}{N-2}\bar{Y}+\frac{1}{N-2}[\bar{\sigma}(X)-1][Y(1)+Y(2)]\\
 & =\bar{\sigma}(X)Y(k)+\bar{Y}+\text{O}(N^{-3/2})=\bar{\sigma}(Y)Y(k)+\bar{Y}+\text{O}(N^{-3/2}).
\end{aligned}
\]
Define
\begin{equation}
Sy=\big(\bar{\sigma}(y)y_{1}+\bar{y},\bar{\sigma}(y)y_{2}+\bar{y},y_{3},\dots,y_{N}\big),\quad y\in\mathbb{R}^{N}.\label{eq:-69}
\end{equation}
Since $X(k)=Y(k)$, $k>2$, we see that
\begin{equation}
X=SY+\text{O}(N^{-3/2}).\label{eq:-68}
\end{equation}
Let $p(x)=(2\pi)^{-N/2}e^{-|x|^{2}/2}$ be the density function of
the $N$-dimensional standard normal distribution. By (\ref{eq:-66}),
(\ref{eq:-68}), and the smoothness of $p(x)$,
\[
\begin{aligned}E^{\prime} & =\mathbb{E}[f(Y(1))f(Y(2))]\\
 & =\mathbb{E}\Big[f(Y(1))f(Y(2))\frac{\bar{\sigma}(Y)^{2}}{\bar{\sigma}(X)^{2}}\frac{p(SY)}{p(X)}\Big]+\text{O}(N^{-3/2})\\
 & =\int_{\mathbb{R}^{N}}f[(Tx)_{1}]f[(Tx)_{2}]\frac{\bar{\sigma}(Tx)^{2}}{\bar{\sigma}(x)^{2}}p(STx)dx+\text{O}(N^{-3/2}).
\end{aligned}
\]
Moreover, by (\ref{eq:-67}) and change of variable,
\begin{equation}
\begin{aligned}E^{\prime} & =\int_{\mathbb{R}^{N}}f(y_{1})f(y_{2})\frac{\bar{\sigma}(y)^{2}}{\bar{\sigma}(T^{-1}y)^{2}}p(Sy)\det[\partial T^{-1}(y)]dy+\text{O}(N^{-3/2})\\
 & =\int_{\mathbb{R}^{N}}f(y_{1})f(y_{2})p(Sy)\bar{\sigma}(y)^{2}[1-N^{-1}(2+y_{1}^{2}+y_{2}^{2})]^{-1}dy+\text{O}(N^{-3/2})\\
 & =\int_{\mathbb{R}^{N}}f(y_{1})f(y_{2})\bar{\sigma}(y)^{2}[1+2N^{-1}(1+y_{1}^{2})]p(Sy)dy+\text{O}(N^{-3/2})\\
 & =\int_{\mathbb{R}^{N}}f(x_{1})f(x_{2})\bar{\sigma}(x)^{2}[1+2N^{-1}(1+x_{1}^{2})]\frac{p(Sx)}{p(x)}p(x)dx+\text{O}(N^{-3/2})\\
 & =\mathbb{E}\Big[f(X(1))f(X(2))\bar{\sigma}(X)^{2}[1+2N^{-1}(1+X(1)^{2})]\frac{p(SX)}{p(X)}\Big]+\text{O}(N^{-3/2}).
\end{aligned}
\label{eq:-70}
\end{equation}
Similarly,
\begin{equation}
\begin{aligned}\frac{1}{N}E & =\frac{1}{N}\mathbb{E}\Big[f(X(1))^{2}\bar{\sigma}(X)^{2}\frac{p(SX)}{p(X)}\Big]+\text{O}(N^{-2})\\
 & =\frac{1}{N}\mathbb{E}\Big[f(X(1))^{2}\frac{p(SX)}{p(X)}\Big]+\text{O}(N^{-3/2}).
\end{aligned}
\label{eq:-81}
\end{equation}
By (\ref{eq:-69}), $\partial p(x)=-xp(x)$, $\partial^{2}p(x)=(xx^{\text{T}}-I)p(x)$,
and Taylor's formula,
\begin{equation}
\begin{aligned}\frac{p(SX)}{p(X)} & =1-\sum_{k=1}^{2}[(\bar{\sigma}-1)X(k)+\bar{X}]X(k)+\frac{1}{2}\sum_{k=1}^{2}[X(k)^{2}-1][(\bar{\sigma}-1)X(k)+\bar{X}]^{2}\\
 & \quad+X(1)X(2)[(\bar{\sigma}-1)X(1)+\bar{X}][(\bar{\sigma}-1)X(2)+\bar{X}]+\text{O}_{2}(N^{-3/2}),
\end{aligned}
\label{eq:-71}
\end{equation}
where, and for the rest of the proof, we denote $\bar{\sigma}=\bar{\sigma}(X)$
for simplicity. Note that $p(SX)/p(X)=1+\text{O}(N^{-1/2})$ as a
consequence of (\ref{eq:-71}). It follows from (\ref{eq:-81}) that
\begin{equation}
\frac{1}{N}E=\frac{1}{N}\mathbb{E}[f(X(1))^{2}]+\text{O}(N^{-3/2})=\frac{1}{N}\text{Var}[f(X)].\label{eq:-82}
\end{equation}
By (\ref{eq:-70}), (\ref{eq:-71}), $\bar{\sigma}=1+\text{O}(N^{-1/2})$
and $p(SX)/p(X)=1+\text{O}(N^{-1/2})$,
\begin{equation}
\begin{aligned}E^{\prime} & =\mathbb{E}\Big[f(X(1))f(X(2))\bar{\sigma}^{2}\frac{p(SX)}{p(X)}\Big]\\
 & \quad+2N^{-1}\mathbb{E}[f(X(1))f(X(2))(1+X(1)^{2})]+\text{O}(N^{-3/2})\\
 & =\mathbb{E}\Big[f(X(1))f(X(2))\bar{\sigma}^{2}\frac{p(SX)}{p(X)}\Big]+\text{O}(N^{-3/2})\\
 & =\mathbb{E}\Big[f(X(1))f(X(2))\frac{p(SX)}{p(X)}\Big]\\
 & \quad+\mathbb{E}\Big[f(X(1))f(X(2))(\bar{\sigma}^{2}-1)\frac{p(SX)}{p(X)}\Big]+\text{O}(N^{-3/2})\\
 & =E_{1}^{\prime}+E_{2}^{\prime}+\text{O}(N^{-3/2}).
\end{aligned}
\label{eq:-72}
\end{equation}
For $E_{1}^{\prime}$, by $\mathbb{E}[f(X)]=0$, (\ref{eq:-71}) and
\[
\begin{aligned}(\bar{\sigma}-1)^{2} & =(\bar{\sigma}+1)^{-2}(\bar{\sigma}^{2}-1)^{2}=\frac{1}{4}(\bar{\sigma}^{2}-1)^{2}+\text{O}(N^{-3/2}),\\
(\bar{\sigma}-1)\bar{X} & =(\bar{\sigma}+1)^{-1}(\bar{\sigma}^{2}-1)\bar{X}=\frac{1}{2}(\bar{\sigma}^{2}-1)\bar{X}+\text{O}(N^{-3/2}),
\end{aligned}
\]
we deduce that
\[
\begin{aligned}E_{1}^{\prime} & =-2\mathbb{E}[X(1)^{2}f(X(1))f(X(2))(\bar{\sigma}-1)]-2\mathbb{E}[X(1)f(X(1))f(X(2))\bar{X}]\\
 & \quad+\frac{1}{4}\mathbb{E}[(X(1)^{2}-1)X(1)^{2}f(X(1))f(X(2))(\bar{\sigma}^{2}-1)^{2}]\\
 & \quad+\mathbb{E}[(X(1)^{2}-1)f(X(1))f(X(2))\bar{X}^{2}]\\
 & \quad+\mathbb{E}[(X(1)^{2}-1)X(1)f(X(1))f(X(2))(\bar{\sigma}^{2}-1)\bar{X}]\\
 & \quad+\frac{1}{4}\mathbb{E}[X(1)^{2}f(X(1))X(2)f(X(2))(\bar{\sigma}^{2}-1)^{2}]\\
 & \quad+\mathbb{E}[X(1)f(X(1))X(2)f(X(2))\bar{X}^{2}]\\
 & \quad+\mathbb{E}[X(1)^{2}f(X(1))X(2)f(X(2))(\bar{\sigma}^{2}-1)\bar{X}]+\text{O}(N^{-3/2}).
\end{aligned}
\]
It follows from the above and Lemma \ref{lem:-4} that
\[
\begin{aligned}E_{1}^{\prime} & =-\frac{1}{N}\mathbb{E}[Xf(X)]^{2}-\frac{1}{2N}\mathbb{E}[(X^{2}-1)f(X)]^{2}+\text{O}(N^{-3/2}).\end{aligned}
\]
For $E_{2}^{\prime}$, by (\ref{eq:-71}) again,
\[
\begin{aligned}E_{2}^{\prime} & =-2\mathbb{E}[f(X(1))f(X(2))(\bar{\sigma}^{2}-1)(\bar{\sigma}-1)X(1)^{2}]\\
 & \quad-2\mathbb{E}[X(1)f(X(1))f(X(2))(\bar{\sigma}^{2}-1)\bar{X}]+\text{O}(N^{-3/2})\\
 & =-\mathbb{E}[X(1)^{2}f(X(1))f(X(2))(\bar{\sigma}^{2}-1)^{2}]\\
 & \quad-2\mathbb{E}[X(1)f(X(1))f(X(2))(\bar{\sigma}^{2}-1)\bar{X}]+\text{O}(N^{-3/2}).
\end{aligned}
\]
By Lemma \ref{lem:-4} and $\mathbb{E}[f(X)]=0$, we deduce that $\begin{aligned}E_{2}^{\prime} & =\text{O}(N^{-3/2})\end{aligned}
$. Therefore, 
\begin{equation}
\begin{aligned}E_{1} & =-\frac{1}{N}\mathbb{E}[Xf(X)]^{2}-\frac{1}{2N}\mathbb{E}[(X^{2}-1)f(X)]^{2}+\text{O}(N^{-3/2}).\end{aligned}
\label{eq:-83}
\end{equation}
By (\ref{eq:-84}), (\ref{eq:-81}), (\ref{eq:-83}),
\begin{equation}
\mathbb{E}\Big[\Big(\frac{1}{N}\sum_{k=1}^{N}f(\tilde{X}^{(2)}(k))\Big)^{2}\Big]=\frac{A_{N}}{N}\Big(\text{Var}[f(X)]-\mathbb{E}[Xf(X)]^{2}-\frac{1}{2}\mathbb{E}[(X^{2}-1)f(X)]^{2}\Big).\label{eq:-87}
\end{equation}
Similar argument shows that 
\begin{equation}
\mathbb{E}\Big(\frac{1}{N}\sum_{k=1}^{N}f(\tilde{X}^{(2)}(k))\Big)=\text{O}(N^{-1}).\label{eq:-88}
\end{equation}
Combining (\ref{eq:-87}) and (\ref{eq:-88}) proves (\ref{eq:-65})
for $n=1$ and $\text{Var}(X)=1$. 

The proof for $n>1$ and $\text{Var}(X)=I$ is similar to the above.
For $n>1$ and correlated normal distributions, the conclusion follows
from change of variable $Y=\Sigma^{-1/2}X$ and the identity 
\[
\mathrm{tr}\big[\mathbb{E}\big((\Sigma^{-1/2}XX^{\mathrm{T}}\Sigma^{-1/2}-I)f(X)\big)\big)^{2}\big]=\mathrm{tr}\big[\mathbb{E}\big((\Sigma^{-1}XX^{\mathrm{T}}-I)f(X)\big)\big)^{2}\big].
\]
 
\end{proof}

\section{\label{sec:-2}Numerical results}

In this section we present some numerical experiments to support the
results in this paper. The C++ source code for numerical experiments
in this section is available in the GitHub project ``umm''\footnote{Available at \href{https://github.com/liuxuan1111/umm}{https://github.com/liuxuan1111/umm}}. 

Table \ref{tab:} shows the pricing results of a down-and-in put option---a
typical types of discontinuous payoff---on a single asset ($n=1$).
The market parameters used are: volatility $\sigma=0.3$, interest
rate $r=0.05$, stock spot $S=1.0$, strike $K=1.0$, maturity $T=1.0$,
and knock-in barrier $B=0.8$. In Table \ref{tab:}, ``IID'' denotes
the plain Monte Carlo method, while ``MM1'' and ``MM2'' refer
to the first and the second order moment matching Monte Carlo, respectively.
``SE(MM1)'' and ``SE(MM2)'' stand for the standard deviation estimated
using (\ref{eq:-27}) and (\ref{eq:-65}), whereas ``SE(MMS1)''
and ``SE(MMS2)'' are the standard deviation estimated from 500 independent
sets of moment matching Monte Carlo simulations, each with $N$ samples.
The results clearly show that the variance of the second order moment
matching Monte Carlo is consistently smaller than that of its first
order counterpart, which in turn is smaller than that of the plain
Monte Carlo. Moreover, the variance of moment matching estimates computed
by (\ref{eq:-27}) and (\ref{eq:-65}) align closely with estimations
derived from multiple independent simulations of $N$ samples. Figure
\ref{fig:} compares the standard errors in a $\log$--$\log$ plot,
where ``mm1\_seed'' and ``mm2\_seed'' correspond to ``SE(MMS1)''
and ``SE(MMS2)'' in Table \ref{tab:}.

Table \ref{tab:-1} reports the pricing results of down-and-in put
option on a worst-of basket consisting of $n=3$ assets. That is,
the basket performance is computed as the worst of the three stock
performances. The market parameters used are: volatilities $\sigma=(0.3,0.2,0.4)$,
interest rate $r=0.05$, stock spots $S=(1.0,1.0,1.0)$, strike $K=1.0$,
maturity $T=1.0$, knock-in barrier $B=0.8$, and the correlation
matrix
\[
\rho=\left[\begin{array}{ccc}
1 & 0.3 & 0.1\\
0.3 & 1 & 0.5\\
0.1 & 0.5 & 1
\end{array}\right].
\]
The definitions of columns in Table \ref{tab:-1} are identical to
those in Table \ref{tab:}. Figure \ref{fig:-1}, analogous to Figure
\ref{fig:}, presents the standard errors in a $\log$-$\log$ plot.

\begin{table}[H]

\centering{}%
\begin{tabular}{|c|c|c|c|c|c|c|c|c|}
\hline 
\textbf{\small{}$N$} & {\small{}PV(IID)} & {\small{}PV(MM1)} & {\small{}PV(MM2)} & {\small{}SE(IID)} & {\small{}SE(MM1)} & {\small{}SE(MMS1)} & {\small{}SE(MM2)} & {\small{}SE(MMS2)}\tabularnewline
\hline 
{\small{}10000} & {\small{}0.06932} & {\small{}0.06903} & {\small{}0.06830} & {\small{}0.00135} & {\small{}0.00087} & {\small{}0.00084} & {\small{}0.00051} & {\small{}0.00053}\tabularnewline
\hline 
{\small{}20000} & {\small{}0.06882} & {\small{}0.06895} & {\small{}0.06814} & {\small{}0.00095} & {\small{}0.00062} & {\small{}0.00061} & {\small{}0.00036} & {\small{}0.00038}\tabularnewline
\hline 
{\small{}40000} & {\small{}0.06807} & {\small{}0.06837} & {\small{}0.06801} & {\small{}0.00067} & {\small{}0.00044} & {\small{}0.00045} & {\small{}0.00026} & {\small{}0.00027}\tabularnewline
\hline 
{\small{}80000} & {\small{}0.06798} & {\small{}0.06828} & {\small{}0.06803} & {\small{}0.00047} & {\small{}0.00031} & {\small{}0.00032} & {\small{}0.00019} & {\small{}0.00019}\tabularnewline
\hline 
{\small{}160000} & {\small{}0.06787} & {\small{}0.06804} & {\small{}0.06795} & {\small{}0.00033} & {\small{}0.00022} & {\small{}0.00022} & {\small{}0.00014} & {\small{}0.00014}\tabularnewline
\hline 
{\small{}320000} & {\small{}0.06780} & {\small{}0.06807} & {\small{}0.06803} & {\small{}0.00024} & {\small{}0.00016} & {\small{}0.00016} & {\small{}9.5e-5} & {\small{}9.4e-05}\tabularnewline
\hline 
{\small{}640000} & {\small{}0.06805} & {\small{}0.06815} & {\small{}0.06806} & {\small{}0.00019} & {\small{}0.00012} & {\small{}0.00013} & {\small{}7.6e-05} & {\small{}7.5e-05}\tabularnewline
\hline 
\end{tabular}\caption{Pricing of down-in put ($n=1$)}
\label{tab:}
\end{table}

\begin{table}[H]
\centering{}%
\begin{tabular}{|c|c|c|c|c|c|c|c|c|}
\hline 
\textbf{\small{}$N$} & {\small{}PV(IID)} & {\small{}PV(MM1)} & {\small{}PV(MM2)} & {\small{}SE(IID)} & {\small{}SE(MM1)} & {\small{}SE(MMS1)} & {\small{}SE(MM2)} & {\small{}SE(MMS2)}\tabularnewline
\hline 
{\small{}10000} & {\small{}0.16226} & {\small{}0.16257} & {\small{}0.16139} & {\small{}0.00187} & {\small{}0.00111} & {\small{}0.00117} & {\small{}0.00073} & {\small{}0.00072}\tabularnewline
\hline 
{\small{}20000} & {\small{}0.16108} & {\small{}0.16275} & {\small{}0.16180} & {\small{}0.00131} & {\small{}0.00080} & {\small{}0.00080} & {\small{}0.00053} & {\small{}0.00053}\tabularnewline
\hline 
{\small{}40000} & {\small{}0.16078} & {\small{}0.16183} & {\small{}0.16179} & {\small{}0.00092} & {\small{}0.00057} & {\small{}0.00055} & {\small{}0.00038} & {\small{}0.00037}\tabularnewline
\hline 
{\small{}80000} & {\small{}0.16126} & {\small{}0.16194} & {\small{}0.16172} & {\small{}0.00065} & {\small{}0.00041} & {\small{}0.00039} & {\small{}0.00027} & {\small{}0.00027}\tabularnewline
\hline 
{\small{}160000} & {\small{}0.16174} & {\small{}0.16178} & {\small{}0.16169} & {\small{}0.00046} & {\small{}0.00029} & {\small{}0.00028} & {\small{}0.00019} & {\small{}0.00018}\tabularnewline
\hline 
{\small{}320000} & {\small{}0.16164} & {\small{}0.16180} & {\small{}0.16175} & {\small{}0.00033} & {\small{}0.00020} & {\small{}0.00020} & {\small{}0.00013} & {\small{}0.00013}\tabularnewline
\hline 
{\small{}640000} & {\small{}0.16162} & {\small{}0.16176} & {\small{}0.16168} & {\small{}0.00026} & {\small{}0.00016} & {\small{}0.00016} & {\small{}0.00011} & {\small{}0.00011}\tabularnewline
\hline 
\end{tabular}\caption{Pricing of down-in put ($n=3$)}
\label{tab:-1}
\end{table}

\begin{figure}[H]
\begin{centering}
\includegraphics[scale=0.7]{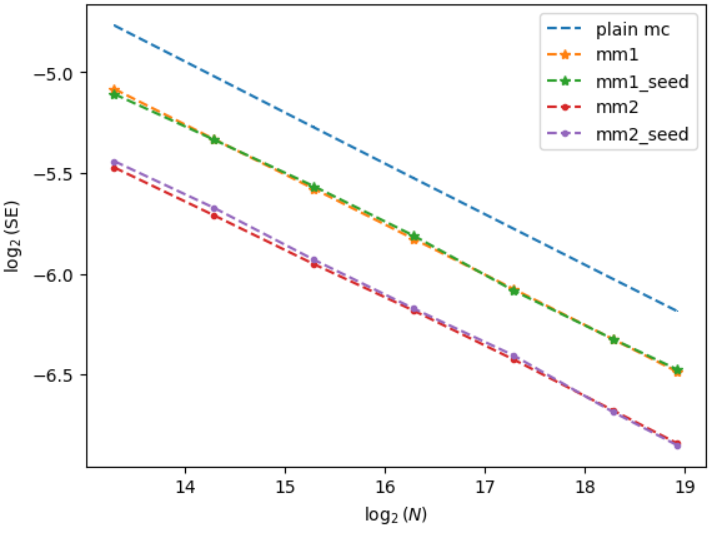}
\par\end{centering}
\caption{$\log$-$\log$ plot of standard error ($n=1$)}
\label{fig:}

\end{figure}

\begin{figure}[H]
\begin{centering}
\includegraphics[scale=0.7]{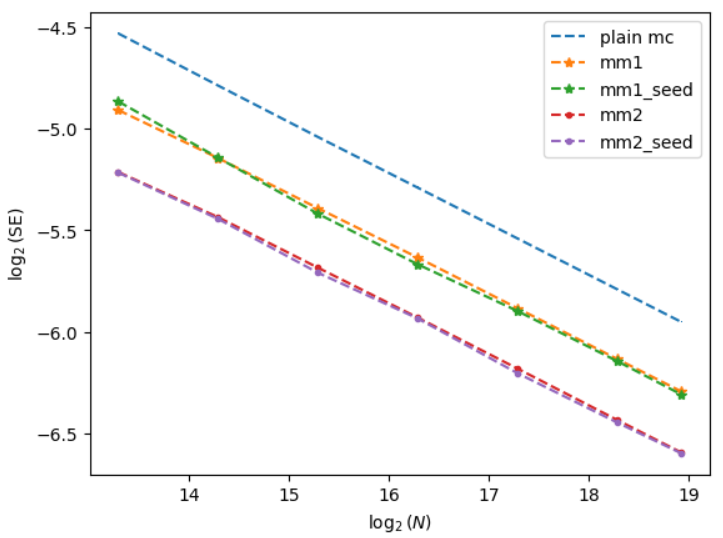}
\par\end{centering}
\caption{$\log$-$\log$ plot of standard error ($n=3$)}
\label{fig:-1}
\end{figure}

\section{\label{sec:-7}Conclusions}

We investigated the conditions under which moment matching Monte Carlo
achieves asymptotically smaller variance than the plain Monte Carlo
for general integrand functions. This asymptotic variance reduction
property is non-trivial: as demonstrated in Example \ref{exa:-2},
moment matching can, in some cases, yield even larger variance. We
resolve this problem by showing that this property holds if and only
if the underlying random distribution is a normal distribution (Theorem
\ref{thm:} and Theorem \ref{thm:-1}). Furthermore, when the underlying
distribution is a normal distribution, the unique distribution satisfying
the universal moment matching property, we derived variance formulae
(Proposition \ref{eq:-27} and Proposition \ref{prop:-2}) which allow
efficient simulation error estimation as by-products of the Monte
Carlo simulation process. As an application, we introduce a non-linear
moment matching scheme (Corollary \ref{cor:}) for general continuous
underlying random distributions. This scheme offers four advantages:
it is easy-to-implement; it guarantees asymptotic variance reduction;
it does not require knowledge on the integrand function; and it supports
efficient estimation of the simulation error. We should remark that
these benefits come at a cost: a modest increase in the computational
expense of random sample generation.

\appendix

\section{\label{sec:-4}Appendix 1}

We give a proof of the necessity half for multi-dimensional case of
Theorem \ref{thm:} without assuming (\ref{eq:-3}) holds for all
bounded smooth functions. It suffices to prove the following proposition.
\begin{prop}
\label{prop:}Let $X=(X_{1},\dots,X_{n})^{\mathrm{T}}$ be a continuous
random vector with $\mathbb{E}(|X|^{2})<\infty$ and differentiable
density $p(x)$. If
\begin{equation}
\sum_{i=1}^{n}\mathbb{E}[\partial_{i}f(X)]\mathbb{E}[X_{i}f(X)]\ge0,\label{eq:-36}
\end{equation}
for any smooth function $f$ with compact support in $\mathbb{R}^{n}$,
then $X$ is a normal distribution with zero mean.
\end{prop}
\begin{proof}
By (\ref{eq:-36}) and integrating by parts,
\begin{equation}
\sum_{i=1}^{n}\int f(x)\partial_{i}p(x)dx\cdot\int f(x)x_{i}p(x)dx\le0,\label{eq:-37}
\end{equation}
for any smooth function $f$ with compact support $\text{supp}(f)\subseteq\text{supp}(p)^{o}$.
Let $u_{k}=(u_{1k},\dots,u_{nk})\in\text{supp}(p)$, $k=1,\dots,n-1$
be $n-1$ (fixed) points such that the matrix
\[
U_{1}=\left[\begin{array}{cccc}
u_{21} & u_{22} & \cdots & u_{2,n-1}\\
u_{31} & u_{32} & \cdots & u_{3,n-1}\\
\cdots & \cdots & \cdots & \cdots\\
u_{n1} & u_{n2} & \cdots & u_{n,n-1}
\end{array}\right],
\]
is non-singular. Let $\phi(x)$ be a smooth function supported in
the unit ball $\{x:|x|<1\}$ and $\int\phi(x)dx=1$. Let $\phi_{\epsilon}(x)=\epsilon^{-n}\phi(x/\epsilon)$
for $\epsilon>0$. For any $x_{1}=(x_{11},\dots,x_{n1}),x_{2}=(x_{12},\dots,x_{n2})\in\text{supp}(p)$
and any $a_{1},a_{2},b_{1},\dots,b_{n-1}\in\mathbb{R}$, setting
\[
f(x)=\sum_{j=1}^{2}a_{j}\phi_{\epsilon}(x_{j}-x)+\sum_{k=1}^{n-1}b_{k}\phi_{\epsilon}(u_{k}-x),
\]
in (\ref{eq:-37}) and letting $\epsilon\to0$ gives
\begin{equation}
\sum_{i=1}^{n}\Big(\sum_{j=1}^{2}a_{j}\partial_{i}p(x_{j})+\sum_{k=1}^{n-1}b_{k}\partial_{i}p(u_{k})\Big)\Big(-\sum_{j=1}^{2}a_{j}x_{ij}p(x_{j})-\sum_{k=1}^{n-1}b_{k}u_{ik}p(u_{k})\Big)\ge0.\label{eq:-38}
\end{equation}
Now let $(b_{1},\dots,b_{n-1})$ be the solution of the system
\begin{equation}
\sum_{j=1}^{2}a_{j}x_{ij}p(x_{j})+\sum_{k=1}^{n-1}b_{k}u_{ik}p(u_{k})=0,\quad i=2,\dots,n,\label{eq:-40}
\end{equation}
that is,
\[
\left[\begin{array}{c}
b_{1}\\
b_{2}\\
\dots\\
b_{n-1}
\end{array}\right]=-P^{-1}\left[\begin{array}{cc}
x_{21}p(x_{1}) & x_{22}p(x_{2})\\
x_{31}p(x_{1}) & x_{32}p(x_{2})\\
\dots & \dots\\
x_{n1}p(x_{1}) & x_{n2}p(x_{2})
\end{array}\right]\left[\begin{array}{c}
a_{1}\\
a_{2}
\end{array}\right],
\]
where 
\[
P=U_{1}\text{diag}[p(u_{1}),p(u_{2}),\dots,p(u_{n-1})].
\]
By (\ref{eq:-38}) and (\ref{eq:-40}),
\begin{equation}
\Big[\sum_{j=1}^{2}a_{j}\Big(\partial_{1}p(x_{j})-\sum_{k=2}^{n}\alpha_{k}x_{kj}p(x_{j})\Big)\Big]\Big[\sum_{j=1}^{2}a_{j}\Big(-x_{1j}p(x_{j})+\sum_{k=2}^{n}\beta_{k}x_{kj}p(x_{j})\Big)\Big]\ge0,\label{eq:-39}
\end{equation}
for any $a_{1},a_{2}\in\mathbb{R}$, where 
\[
(\alpha_{2},\dots,\alpha_{n})=(\partial_{1}p(u_{1}),\dots,\partial_{1}p(u_{n-1}))P^{-1},
\]
and 
\[
(\beta_{2},\dots,\beta_{n})=(u_{11}p(u_{1}),\dots,u_{1,n-1}p(u_{n-1}))P^{-1}.
\]
By (\ref{eq:-39}) and Lemma \ref{lem:-2},
\[
\dfrac{\partial_{1}p(x_{1})-\sum_{k=2}^{n}\alpha_{k}x_{k1}p(x_{1})}{-x_{11}p(x_{1})+\sum_{k=2}^{n}\beta_{k}x_{k1}p(x_{1})}=\dfrac{\partial_{1}p(x_{2})-\sum_{k=2}^{n}\alpha_{k}x_{k2}p(x_{2})}{-x_{12}p(x_{2})+\sum_{k=2}^{n}\beta_{k}x_{k2}p(x_{2})}.
\]
Therefore, there exists a constant $c$ such that
\[
\dfrac{\partial_{1}p(x)-\sum_{k=2}^{n}\alpha_{k}x_{k}p(x)}{-x_{1}p(x)+\sum_{k=2}^{n}\beta_{k}x_{k}p(x)}=c,\quad x\in\text{supp}(p).
\]
Equivalently, there exist constants $c_{11},\dots,c_{1n}$ such that
\[
\partial_{1}p(x)=-\sum_{k=1}^{n}c_{1k}x_{k}p(x),\quad x\in\text{supp}(p).
\]
Similarly, there exist constants $c_{ij},1\le i,j\le n$ such that
\begin{equation}
\partial_{i}p(x)=-\sum_{k=1}^{n}c_{ik}x_{k}p(x),\quad i=1,2,\dots,n,\label{eq:-41}
\end{equation}
for any $x\in\text{supp}(p)$. Solving the differential equation system
(\ref{eq:-41}) gives that $p(x)=\lambda e^{-\frac{1}{2}x^{\text{T}}Cx}$.
Now Example \ref{exa:} and (\ref{eq:-61}) imply that $\partial p\not=0$;
that is $C\not=0$. It follows from $\int p(x)dx=1$ that $C$ is
positive definite and $\lambda=[2\pi\det(C^{-1})]^{-n/2}$. Therefore,
$X$ is a normal distribution with zero mean.
\end{proof}

\section{\label{sec:-5}Appendix 2}

We give a proof of Lemma \ref{lem:-4} in Section \ref{sec:-3}.
\begin{proof}[Proof of Lemma \ref{lem:-4}]
Equation (\ref{eq:-73}) follows readily from the fact that $\mathbb{E}[g(X(1),X(2))X_{i}(k)X_{j}(l)]=0$
if $k\not=l$ or any $k>2$ or $l>2$. For (\ref{eq:-74}), by $\mathbb{E}(X_{i}X_{j})=\delta_{ij}$
and (\ref{eq:-73}),
\[
\begin{aligned} & \mathbb{E}[g(X(1),X(2))(\bar{\Sigma}-I)_{ij}]\\
 & =\frac{1}{N}\sum_{k=1}^{N}\mathbb{E}[g(X(1),X(2))(X_{i}(k)X_{j}(k)-\delta_{ij})]-\mathbb{E}[g(X(1),X(2))\bar{X}_{i}\bar{X}_{j}]\\
 & =\frac{1}{N}\sum_{k=1}^{2}\mathbb{E}[g(X(1),X(2))(X_{i}(k)X_{j}(k)-\delta_{ij})]-\frac{1}{N}\mathbb{E}[g(X(1),X(2))]\delta_{ij}+\mathrm{O}(N^{-2})\\
 & =\frac{1}{N}\sum_{k=1}^{2}\mathbb{E}[g(X(1),X(2))X_{i}(k)X_{j}(k)]-\frac{3}{N}\mathbb{E}[g(X(1),X(2))]\delta_{ij}+\mathrm{O}(N^{-2}).
\end{aligned}
\]
For (\ref{eq:-75}), by $\bar{\Sigma}-I=\frac{1}{N}\sum_{k=1}^{N}X(k)X(k)^{\text{T}}-I+\text{O}(N^{-1})$
and $\bar{X}=\text{O}(N^{-1/2})$,
\[
\begin{aligned} & \mathbb{E}[g(X(1),X(2))(\bar{\Sigma}-I)_{ij}\bar{X}_{p}]\\
 & =\frac{1}{N^{2}}\sum_{1\le k,l\le N}\mathbb{E}[g(X(1),X(2))(X_{i}(k)X_{j}(k)-\delta_{ij})X_{p}(l)]+\text{O}(N^{-3/2})\\
 & =\frac{1}{N^{2}}\sum_{1\le k,l\le2}\mathbb{E}[g(X(1),X(2))(X_{i}(k)X_{j}(k)-\delta_{ij})X_{p}(l)]\\
 & \quad+\frac{1}{N^{2}}\sum_{2<k\le N}\mathbb{E}[g(X(1),X(2))(X_{i}(k)X_{j}(k)-\delta_{ij})X_{p}(k)]+\text{O}(N^{-3/2})\\
 & =\frac{N-2}{N^{2}}\mathbb{E}[g(X(1),X(2))]\mathbb{E}[(X_{i}(k)X_{j}(k)-\delta_{ij})X_{p}(k)]+\text{O}(N^{-3/2})\\
 & =\frac{1}{N}\mathbb{E}[g(X(1),X(2))]\mathbb{E}(X_{i}X_{j}X_{p})+\mathrm{O}(N^{-3/2}).
\end{aligned}
\]
Similarly, for (\ref{eq:-76}),
\[
\begin{aligned} & \mathbb{E}[g(X(1),X(2))(\bar{\Sigma}-I)_{ij}(\bar{\Sigma}-I)_{pq}]\\
 & =\frac{1}{N^{2}}\sum_{1\le k,l\le N}\mathbb{E}[g(X(1),X(2))(X_{i}(k)X_{j}(k)-\delta_{ij})(X_{p}(l)X_{q}(l)-\delta_{pq})]+\text{O}(N^{-3/2})\\
 & =\frac{1}{N^{2}}\sum_{1\le k,l\le2}\mathbb{E}[g(X(1),X(2))(X_{i}(k)X_{j}(k)-\delta_{ij})(X_{p}(l)X_{q}(l)-\delta_{pq})]\\
 & \quad+\frac{1}{N^{2}}\sum_{2<k\le N}\mathbb{E}[g(X(1),X(2))(X_{i}(k)X_{j}(k)-\delta_{ij})(X_{p}(k)X_{q}(k)-\delta_{pq})]+\text{O}(N^{-3/2})\\
 & =\frac{N-1}{N^{2}}\mathbb{E}[g(X(1),X(2))]\mathbb{E}[(X_{i}X_{j}-\delta_{ij})(X_{p}X_{q}-\delta_{pq})]+\text{O}(N^{-3/2})\\
 & =\frac{1}{N}\mathbb{E}[g(X(1),X(2))][\mathbb{E}(X_{i}X_{j}X_{p}X_{q})-\delta_{ij}\delta_{pq}]+\text{O}(N^{-3/2}).
\end{aligned}
\]

Suppose in addition that $\mathbb{E}[g(X,y)]=0$ for any $y\in\mathbb{R}^{n}$.
By the central limit theorem,
\begin{equation}
\frac{1}{N-1}\sum_{k=2}^{N}g(X(1),X(k))=\text{O}(N^{-1/2}).\label{eq:-78}
\end{equation}
Note that, by $\bar{\Sigma}=I+\text{O}(N^{-1/2})$ and $\bar{\Sigma}^{1/2}=I+\text{O}(N^{-1/2})$,
\begin{equation}
\bar{\Sigma}^{1/2}-I=(\bar{\Sigma}^{1/2}+I)^{-1}(\bar{\Sigma}-I)=\frac{1}{2}(\bar{\Sigma}-I)+\text{O}(N^{-1}).\label{eq:-79}
\end{equation}
Therefore, by (\ref{eq:-78}) and (\ref{eq:-79}),
\begin{equation}
\frac{1}{N-1}\sum_{k=2}^{N}g(X(1),X(k))(\bar{\Sigma}^{1/2}-I)=\frac{1}{2(N-1)}\sum_{k=2}^{N}g(X(1),X(k))(\bar{\Sigma}-I)+\text{O}_{2}(N^{-3/2}).\label{eq:-80}
\end{equation}
Moreover, by symmetry and (\ref{eq:-80}),
\[
\begin{aligned} & \mathbb{E}[g(X(1),X(2))(\bar{\Sigma}^{1/2}-I)_{ij}]\\
 & =\frac{1}{N-1}\sum_{k=2}^{N}\mathbb{E}[g(X(1),X(k))(\bar{\Sigma}^{1/2}-I)_{ij}]\\
 & =\frac{1}{2(N-1)}\sum_{k=2}^{N}\mathbb{E}[g(X(1),X(k))(\bar{\Sigma}-I)_{ij}]+\text{O}(N^{-3/2})\\
 & =\frac{1}{2}\mathbb{E}[g(X(1),X(2))(\bar{\Sigma}-I)_{ij}]+\text{O}(N^{-3/2}).
\end{aligned}
\]
Equation (\ref{eq:-77}) now follows easily from (\ref{eq:-74}) and
$\mathbb{E}[g(X(1),X(2))]=0$.
\end{proof}

\end{document}